\documentclass{amsart}

\usepackage{amsmath, amssymb, amsthm, amsfonts}

\input xy
\xyoption{all}

\usepackage{hyperref}

\swapnumbers
\numberwithin{equation}{section}

\theoremstyle{plain}
\newtheorem{theorem}[subsection]{Theorem}
\newtheorem{lemma}[subsection]{Lemma}
\newtheorem{prop}[subsection]{Proposition}
\newtheorem{cor}[subsection]{Corollary}
\newtheorem{conj}[subsection]{Conjecture}

\newtheorem*{claim}{Claim}

\theoremstyle{definition}
\newtheorem{defn}[subsection]{Definition}

\newtheorem{remark}[subsection]{Remark}

\def\AA{\mathbb{A}}

\def\CC{\mathbb{C}}

\def\GG{\mathbb{G}}

\def\NN{\mathbb{N}}

\def\PP{\mathbb{P}}
\def\QQ{\mathbb{Q}}
\def\RR{\mathbb{R}}

\def\ZZ{\mathbb{Z}}

\newcommand\cB{\mathcal{B}}

\newcommand\cL{\mathcal{L}}

\newcommand\cN{\mathcal{N}}
\newcommand\cO{\mathcal{O}}
\newcommand\cP{\mathcal{P}}

\newcommand\cS{\mathcal{S}}
\newcommand\cT{\mathcal{T}}
\newcommand\cU{\mathcal{U}}

\newcommand\cX{\mathcal{X}}

\newcommand\cZ{\mathcal{Z}}

\def\bI{\mathbf{I}}

\newcommand\frI{\mathfrak{I}}

\newcommand\fra{\mathfrak{a}}

\newcommand\frc{\mathfrak{c}}
\newcommand\frd{\mathfrak{d}}

\newcommand\frg{\mathfrak{g}}
\newcommand\frh{\mathfrak{h}}

\newcommand\fm{\mathfrak{m}}
\newcommand\frn{\mathfrak{n}}

\newcommand\frt{\mathfrak{t}}

\newcommand\tilW{\widetilde{W}}

\newcommand{\codim}{\textup{codim}}

\newcommand\Cox{\textup{Cox}}

\newcommand\ev{\textup{ev}}

\newcommand{\Fl}{\textup{Fl}}

\newcommand\Gal{\textup{Gal}}

\newcommand{\Gr}{\textup{Gr}}

\renewcommand{\Im}{\textup{Im}}

\newcommand\Irr{\textup{Irr}}

\newcommand\Lag{\textup{Lag}}
\newcommand\Lie{\textup{Lie}\ }

\newcommand{\Nm}{\textup{Nm}}

\newcommand\pr{\textup{pr}}

\newcommand{\red}{\textup{red}}
\newcommand{\reg}{\textup{reg}}

\newcommand\rk{\textup{rk}}

\newcommand\Span{\textup{Span}}
\newcommand\Spec{\textup{Spec}\ }

\newcommand{\Tr}{\textup{Tr}}

\newcommand{\val}{\textup{val}}

\newcommand\Aut{\textup{Aut}}

\newcommand\End{\textup{End}}

\newcommand\GL{\textup{GL}}

\newcommand\SL{\textup{SL}}

\newcommand\SO{\textup{SO}}

\newcommand\so{\mathfrak{so}}
\newcommand\Spin{\textup{Spin}}
\newcommand\Sp{\textup{Sp}}

\newcommand{\Gm}{\GG_m}

\newcommand{\ad}{\textup{ad}}
\newcommand{\Ad}{\textup{Ad}}

\newcommand\xcoch{\mathbb{X}_*}

\newcommand{\incl}{\hookrightarrow}
\newcommand{\isom}{\stackrel{\sim}{\to}}
\newcommand{\bij}{\leftrightarrow}
\newcommand{\surj}{\twoheadrightarrow}

\renewcommand{\j}[1]{\langle{#1}\rangle}
\newcommand{\wt}[1]{\widetilde{#1}}

\newcommand\quash[1]{}
\newcommand\un{\underline}

\newcommand{\ov}{\overline}
\newcommand{\bs}{\backslash}
\newcommand{\tl}[1]{[\![#1]\!]}
\newcommand{\lr}[1]{(\!(#1)\!)}
\newcommand\sss{\subsubsection}

\newcommand\op{\oplus}
\newcommand\ot{\otimes}

\newcommand{\sslash}{\mathbin{/\mkern-6mu/}}
\newcommand\vn{\varnothing}
\renewcommand\c{\circ}
\newcommand\mj{\jmath}

\newcommand{\cohog}[2]{\textup{H}^{#1}({#2})}     

\renewcommand\a\alpha
\renewcommand\b\beta
\newcommand\g\gamma
\newcommand\G\Gamma
\renewcommand\d\delta
\newcommand\D\Delta
\newcommand{\e}{\epsilon}
\newcommand{\io}{\iota}

\newcommand{\s}{\sigma}

\newcommand{\z}{\zeta}
\newcommand{\ep}{\epsilon}

\renewcommand{\l}{\lambda}
\renewcommand{\L}{\Lambda}

\newcommand\hs{\heartsuit}

\newcommand\KL{\textup{KL}}
\newcommand\RT{\textup{RT}}

\title{Minimal reduction type and the Kazhdan-Lusztig map}

\dedicatory{Dedicated to the memory of T.A.Springer}
\author{Zhiwei Yun}
\thanks{Supported by the Packard Foundation and the Simons Foundation.}
\address{Department of Mathematics, Massachusetts Institute of Technology, 77 Massachusetts Ave, Cambridge, MA 02139}
\email{zyun@mit.edu}
\date{}

\begin{document}

\begin{abstract}
We introduce the notion of minimal reduction type of an affine Springer fiber, and use it to define a map from the set of conjugacy classes in the Weyl group to the set of nilpotent orbits. We show that this map is the same as the one defined by Lusztig in \cite{Lfromto}, and that the Kazhdan-Lusztig map in \cite{KL} is a section of our map. This settles several conjectures in the literature. For classical groups, we prove more refined results by introducing and studying the ``skeleta'' of affine Springer fibers. 
\end{abstract}

\maketitle

\tableofcontents

\section{Introduction}\label{s:intro}
\subsection{Overview}
Throughout this paper, let $G$ be a connected reductive group over $\CC$ with Lie algebra $\frg$.  Fix a Borel subgroup $B\subset G$ with a maximal torus $T\subset B$. Let $\frt$ denote the Lie algebra of $T$. Let $W=W(G,T)$ be the corresponding Weyl group. Let $\un W$ be the set of conjugacy classes in $W$. The set $\un W$ is canonically independent of the choice of $T$. \footnote{ For any two maximal tori $T_{1}$ and $T_{2}$ of $G$,  one can find $g\in G(\CC)$ such that $gT_{1}g^{-1}=T_{2}$ and it induce an isomorphism $\io_{g}: W(G,T_{1})\isom W(G,T_{2})$; for different choices of $g$, the isomorphisms $\io_{g}$ differ by conjugacy.}

Let $\cN$ be the nilpotent cone in $\frg$. Let $\un \cN$ be the set of $G$-orbits on $\cN$ under the adjoint action. This is a partially ordered set such that $\cO\le\cO'$ if and only if $\cO\subset\ov{\cO'}$. For $\cO\in\un\cN$, let $d_{\cO}$ be the dimension of the Springer fiber of any $e\in \cO$.

The main results of this article can be summarized by saying that the following two diagrams are identical:
\begin{equation*}
\xymatrix{ \un W \ar@<.6ex>[r]^{\RT_{\min}} & \un \cN \ar@<.6ex>[l]^{\KL} & \un W \ar@<.6ex>[r]^{\Phi} & \un \cN \ar@<.6ex>[l]^{\Psi}}
\end{equation*}
Here,  $\RT_{\min}: \un W\to \un \cN$ is a map that we will define in this paper using affine Springer fibers; $\KL$ is the map defined by Kazhdan-Lusztig \cite[\S9.1]{KL} using the loop group. The maps $\Phi$  and $\Psi$ are defined by Lusztig \cite{Lfromto,Lfromto2} using the geometry of $G$ itself and not the loop group.

\subsection{Loop group and affine Grassmannian}Let $\CC\lr{t}$ (resp. $\CC\tl{t}$) be the field (resp. ring) of formal Laurent (resp. Taylor) series in one variable $t$. For an affine scheme $X$ over $\CC$, let $LX$ be the formal loop space of $X$ defined as the ind-scheme with  $R$-points $LX(R):=X(R\lr{t})$ for any $\CC$-algebra $R$. Let $L^{+}X$ be the formal arc space of $X$ defined as the scheme with $R$-points $L^{+}X(R):=X(R\tl{t})$. For more details, see \cite[\S1]{PR}. In particular, we have the loop group $LG$  and the arc group $L^{+}G\subset LG$.  Let $\bI\subset L^{+}G$ be the Iwahori subgroup which is the preimage of $B$ under the reduction map $L^{+}G\to G$.  Let $\Gr=LG/L^{+}G$ be the affine Grassmannian of $G$, and $\Fl=LG/\bI$ be the affine flag variety of $G$. We also have the loop Lie algebra  $L\frg$ whose $\CC$-points are $\frg\ot_{\CC}\CC\lr{t}$, and $L^{+}\frg=\frg\ot_{\CC}\CC\tl{t}$.

\subsection{Affine Springer fibers}
Let $L^{\hs}\frg\subset L\frg$ be the subset of topologically nilpotent and generically regular semisimple elements.  The notation $L^{\hs}(-)$ will denote topologically nilpotent regular semisimple elements in the loop space of various spaces related to $\frg$. In this introduction, the various $L^{\hs}(-)$ are only defined as subsets of $\CC$-points of the relevant ambient spaces; they will be viewed as the sets of $\CC$-points of schemes starting in \S\ref{s:Lcw}.

For $\g\in L^{\hs}\frg$, we recall from \cite[\S0]{KL} the definition of affine Springer fibers:
\begin{eqnarray*}
\Gr_{\g}=\{gL^{+}G\in \Gr|\Ad(g^{-1})\g\in L^{+}\frg\};\\
\Fl_{\g}=\{g\bI\in \Fl|\Ad(g^{-1})\g\in \Lie\bI\}.
\end{eqnarray*}
These are closed sub-ind-schemes of $\Gr$ and $\Fl$ respectively. In this paper, we always equip $\Gr_{\g}$ and $\Fl_{\g}$ with the reduced ind-scheme structure. 

The centralizer $(LG)_{\g}$ is the loop group of a maximal torus of $G_{\CC\lr{t}}$, and it acts on $\Gr_{\g}$ and $\Fl_{\g}$.

\subsection{Loop tori and $\un W$}\label{ss:loop tori}
Recall from \cite[Lemma 2]{KL} that maximal tori of $G_{\CC\lr{t}}$ up to $LG$-conjugacy are in canonical bijection with $\un W$. For $w\in W$, a maximal torus $T_{w}$ of $G_{\CC\lr{t}}$ of type $[w]$ can be abstractly constructed as follows: if $w$ has order $m$ and $\z\in\mu_{m}(\CC)$ is a primitive $m$th root of unity, then $T_{w}$ is the fixed point subgroup of the Weil restriction $R_{\CC\lr{t^{1/m}}/\CC\lr{t}}(T_{\CC\lr{t^{1/m}}})$ under the action that is $w$ on $T$ and $t^{1/m}\mapsto \z t^{1/m}$.

A loop torus in $LG$ is the loop group $LH$ associated to a maximal torus $H$ of $G_{\CC\lr{t}}$, so that $LH(\CC)=H(\CC\lr{t})$. Therefore the loop tori in $LG$ up to $LG$-conjugacy are $\{LT_{w}\}_{[w]\in \un W}$.

We will say $\g\in L^{\hs}\frg$ is of type $[w]\in \un W$ if its centralizer $(LG)_{\g}$ is of type $[w]$ under the above-mentioned bijection.  Let $(L^{\hs}\frg)_{[w]}$ be the set of elements of type $[w]$.

Recall that $[w]\in \un W$ is called {\em elliptic} if $\frt^{w}=\Lie Z_{G}$ for any $w\in [w]$.  Let $\un W_{ell}\subset \un W$ denote the set of elliptic conjugacy classes.


Let $L^{+}T_{w}$ be the parahoric subgroup of the loop torus $LT_{w}$. A concrete description of $L^{+}T_{w}$ will be given in the beginning of \S\ref{s:sk}. Let $L^{+}\frt_{w}$ be the Lie algebra of $L^{+}T_{w}$.  Let $L^{++}\frt_{w}\subset L^{+}\frt_{w}$ be the linear subspace of topologically nilpotent elements. The quotient $L^{+}\frt_{w}/L^{++}\frt_{w}$ can be identified with $\frt^{w}$. Let $L^{\hs}\frt_{w}=L^{\hs}\frg\cap L\frt_{w}\subset L^{++}\frt_{w}$ be the subset of topologically nilpotent generically regular semisimple elements. 

\subsection{The function $\d$}\label{ss:d}
Let $\D:\frg\to\AA^{1}$ be the discriminant function. It is characterized as the unique $\Ad(G)$-invariant regular function whose restriction to $\frt$ takes the form $\D_{\frt}(x)=\prod_{\a}\a(x)$, where $\a$ runs over all roots of $G$ with respect to $T$. 

For $\g\in L^{\hs}\frg$ of type $[w]$, let
\begin{equation}\label{dim Sp}
\d_{\g}=\frac{\val(\D(\g))-(\dim\frt-\dim\frt^{w})}{2}.
\end{equation}
By \cite[last line in p.130]{KL} and \cite{B}, $\dim\Fl_{\g}=\dim\Gr_{\g}=\d_{\g}$.

For $[w]\in \un W$, let
\begin{equation*}
\d_{[w]}=\min\{\d_{\g}|\g\in (L^{\hs}\frg)_{[w]}\}.
\end{equation*}

\begin{defn} An element $\g\in  L^{\hs}\frg$ is called {\em shallow of type $[w]\in \un W$}, if $\g$ is of type $[w]$, and $\d(\g)=\d_{[w]}$.
\end{defn}
For $[w]\in \un W$ we denote the set of shallow elements of type $[w]$ by $(L^{\hs}\frg)^{sh}_{[w]}$. Similarly, for a loop torus $LT_{w}$, let $(L^{\hs}\frt_{w})^{sh}$ be the set of shallow elements in $L^{\hs}\frt_{w}$.

\subsection{The Chevalley base}\label{ss:intro c}
Let $\frc=\frg\sslash G=\frt\sslash W$. Let $\chi:\frg\to \frc$ be the natural map. We still use $\chi$ to denote $L\frg\to L\frc$.  Let $L^{\hs}\frc\subset L^{+}\frc$ denote the subset of topologically nilpotent and generically regular semisimple points. Then $L^{\hs}\frc$ is characterized by $L^{\hs}\frg=\chi^{-1}(L^{\hs}\frc)$. 

For $[w]$, let $(L^{\hs}\frc)_{[w]}$ be the image of $(L^{\hs}\frg)_{[w]}$. Let $(L^{\hs}\frc)^{sh}_{[w]}$ be the image of $(L^{\hs}\frg)^{sh}_{[w]}$. Clearly,  $\chi^{-1}((L^{\hs}\frc)_{[w]})=(L^{\hs}\frg)_{[w]}$, and $\chi^{-1}((L^{\hs}\frc)^{sh}_{[w]})=(L^{\hs}\frg)^{sh}_{[w]}$. We will show in \S\ref{s:Lcw} that $(L^{\hs}\frc)_{[w]}^{sh}$ is a scheme with reasonable properties (it is locally closed of finite presentation in $L^{+}\frc$).

\subsection{The Kazhdan-Lusztig map}
In \cite[\S9]{KL} the authors define a map
\begin{equation*}
\KL=\KL_{G}: \un \cN\to \un W
\end{equation*}
as follows. For $e\in \cN$ and a generic lifting $\wt e\in L^{+}\frg$ with reduction $e$, $\wt e$ is a topologically nilpotent regular semisimple element of type $[w]\in \un W$. Then $\KL(\cO_{e}):=[w]$. Here generic lifting means $\wt e\in e+U$ for some open dense subset $U\subset tL^{+}\frg$.

\subsection{Stratification by reduction type}
Let $\g\in L^{\hs}\frg$. Then we have a reduction map
\begin{equation*}
\ev_{\g}: \Gr_{\g}\to [\cN/G]
\end{equation*}
sending $gL^{+}G$ to $\Ad(g^{-1})\g\mod t\in \cN$, well-defined up to the adjoint action of $G$.

For each nilpotent orbit $\cO\subset \cN$, let $\Gr_{\g,\cO}$ be the preimage of $\cO/G$ under $\ev_{\g}$. Let $\Fl_{\g,\cO}$ be the preimage of $\Gr_{\g,\cO}$ under the natural projection $\Fl_{\g}\to \Gr_{\g}$. Then if $\Gr_{\g,\cO}\ne\vn$,  $\Fl_{\g,\cO}\to \Gr_{\g,\cO}$ is a fibration with fibers isomorphic to the Springer fiber $\cB_{e}$ for $e\in \cO$.

\begin{defn}
\begin{enumerate}
\item For $\g\in  L^{\hs}\frg$, we denote by $\RT(\g)$ the subset of $\un\cN$ consisting of nilpotent orbits $\cO$ such that $\Gr_{\g,\cO}\ne\vn$.
\item We define $\RT_{\min}(\g)\subset \RT(\g)$ to be the set of minimal elements  in $\RT(\g)$ under the partial order inherited from that of $\un\cN$. We call elements in $\RT_{\min}(\g)$ the {\em minimal reduction types} of $\g$. 
\end{enumerate}
\end{defn}

Our first result identifies constructs a map $\un W\to \un\cN$ using minimal reduction types, and we prove that $\KL$ is a section of this map.
 
\begin{theorem}\label{th:WtoN}
\begin{enumerate}
\item\label{Uw} For $[w]\in \un W$, there is a unique nilpotent orbit $\cO$ that is contained in $\RT_{\min}(\g)$ for all $\g\in (L^{\hs}\frg)^{sh}_{[w]}$. We denote this  nilpotent orbit by $\RT_{\min}([w])$.  This defines a map:
\begin{equation*}
\RT_{\min}: \un W\to \un \cN.
\end{equation*}
Moreover, there is an open dense subset  $U_{[w]}\subset (L^{\hs}\frg)^{sh}_{[w]}$, such that $\RT_{\min}(\g)=\{\RT_{\min}([w])\}$ for all $\g\in U_{[w]}$.

\item\label{KLsec} The Kazhdan-Lusztig map $\KL$ is a section of $\RT_{\min}$. In particular, $\KL$ is injective and $\RT_{\min}$ is surjective.
\end{enumerate}
\end{theorem}

\begin{remark}
It is conjectured in \cite[\S9.13]{KL} that $\KL$ is injective. The map $\KL$ has been calculated by Spaltenstein for classical groups \cite{Spa-cl, Spa-d} and in most cases for exceptional groups \cite{Spa-ex}. In the remaining cases he gave conjectural answers  (see  \cite[p.1163]{RYGL} which verifies two more cases in $E_{8}$). In particular, $\KL$ was known to be injective for $G$ of classical type and of type $G_{2}, F_{4}$ and $E_{6}$.  Our argument is independent of Spaltenstein's, and implies that his predictions in \cite{Spa-ex} for the remaining cases are correct.
\end{remark}

The next result identifies the pair of maps $(\RT_{\min}, \KL)$ with the pair of maps $(\Phi,\Psi)$ defined by Lusztig which we now recall. 

\subsection{Lusztig's maps}\label{ss:L}
In \cite{Lfromto} Lusztig defines a map 
\begin{equation*}
\Phi=\Phi_{G}: \un W\to \un \cN
\end{equation*}
as follows. Let $w\in W$ be an element with minimal length within its conjugacy class $[w]$. Lusztig shows in \cite[Theorem 0.4(1)]{Lfromto} that among the unipotent orbits that intersect $BwB$, there is a unique minimal orbit $\cU$, and $\Phi([w])$ is defined to be the nilpotent orbit corresponding to $\cU$. He proves that $\Phi([w])$ depends only on the conjugacy class $[w]$ and not on the choice of a minimal length element.  He proves that $\Phi$ is surjective \cite[Theorem 0.4(2)]{Lfromto} and constructs in \cite[Theorem 0.2]{Lfromto2} a canonical section
\begin{equation*}
\Psi=\Psi_{G}: \un \cN\to \un W.
\end{equation*}
For $\cO\in \un\cN$, he proves that the function $\Phi^{-1}(\cO)\to \ZZ_{\ge0}$ given by $[w]\mapsto \dim\frt^{w}$ achieves its minimum at a unique $[w]$ (the ``most elliptic'' conjugacy class in the fiber of $\Phi$), which is defined to be $\Psi(\cO)$.

\begin{theorem}\label{th:PhiPsi} We have
\begin{enumerate}
\item\label{Phi}  $\RT_{\min}=\Phi$.
\item\label{Psi} $\KL=\Psi$.
\end{enumerate}
\end{theorem}
The equality $\KL=\Psi$  was conjectured by Lusztig in \cite{Lfromto2}.

Before stating further results, we make the following conjectures on the minimal reduction types.

\begin{conj}\label{c:min red type}  For any $\g\in  L^{\hs}\frg$, $\RT_{\min}(\g)$ is a singleton.
\end{conj}

A weaker version of the above conjecture is:
\begin{conj}\label{c:min red shallow}  For $[w]\in \un W$, we have $\RT_{\min}(\g)=\{\RT_{\min}([w])\}$ for all $\g\in (L^{\hs}\frg)^{sh}_{[w]}$.\end{conj}

\begin{conj}\label{c:trans} Let $\cO\in \un\cN$ and $[w]=\KL(\cO)$.  Let $\g\in  (L^{\hs}\frg)^{sh}_{[w]}$. Then the centralizer $(LG)_{\g}$ acts transitively on $\Gr_{\g,\cO}$ (which is non-empty by Theorem \ref{th:WtoN}(1) and discrete by Lemma \ref{l:discrete}). 
\end{conj}

In this paper we prove the following partial results towards these conjectures.

\begin{theorem}\label{th:AC} Let $G$ be a reductive group and $[w]\in \un W$. Assume  $[w]$ contains an element $w=(w_{1},\cdots, w_{\ell})\in W_{1}\times \cdots\times W_{\ell}=W'<W$ in a parabolic subgroup $W'$ such that for each $i=1,\cdots, \ell$, one of the following holds
\begin{itemize}
\item either $W_{i}$ is the Weyl group of a Levi subgroup $M_{i}\subset G$ (containing $T$) of type $A,C$ or $G_{2}$;
\item or $[w_{i}]$ is the Coxeter class in $W_{i}$.
\end{itemize}
Then Conjecture \ref{c:min red type} holds for all elements of type $[w]$.

In particular,  Conjecture \ref{c:min red type} holds for $G$ of types $A,C$ and $G_{2}$.
\end{theorem}

\begin{theorem}\label{th:trans} Let $G$ be a reductive group and $[w]=\KL(\cO)$ for some $\cO\in \un\cN$.
Assume  $[w]$ contains an element $w=(w_{1},\cdots, w_{\ell})\in W_{1}\times \cdots\times W_{\ell}=W'<W$ in a parabolic subgroup $W'$ such that for each $i=1,\cdots, \ell$, one of the following holds
\begin{itemize}
\item either $W_{i}$ is of classical type or $G_{2}$;
\item or $[w_{i}]$ is the Coxeter class in $W_{i}$.
\end{itemize}
Then Conjectures \ref{c:min red shallow} and \ref{c:trans} hold for $G$ and $[w]$.

In particular, Conjectures \ref{c:min red shallow} and \ref{c:trans}  hold for $G$ of classical types  and $G_{2}$.
\end{theorem}

\subsection{Remarks on the proofs} The proofs for these theorems are different for classical groups and for exceptional groups.

We introduce certain subvarieties of the affine Grassmannian called the {\em skeleta} (for any  $G$).  These are fixed points of the neutral components of loop tori on $\Gr$. For classical groups and for $G_{2}$, we are able to describe the skeleta explicitly (see \S\ref{s:AC} and \S\ref{s:BD}), from which we deduce the theorems for these groups. It is an interesting question to give explicit descriptions for skeleta in other exceptional types.

To prove Theorems \ref{th:WtoN} and \ref{th:PhiPsi} for exceptional types, we use the properties of certain subsets of the arc space of the Chevalley base, with Theorem \ref{th:Ow} being the key observation.  The proof of Theorem \ref{th:Ow} uses a flatness theorem of Bouthier, Kazhdan and Varshavsky \cite{BKV} for certain maps between arc spaces. Although the constructions of $\RT_{\min}$ and $\KL$ are logically independent of Lusztig's maps $\Phi$ and $\Psi$, our proof that they are equal in the exceptional case relies on Lusztig's explicit calculation of $\Phi$ in \cite{Lfromto}. 

After several reduction steps, we reduce to treat only elliptic conjugacy classes in $\un W$ in almost simple groups. The actual proofs start in \S\ref{s:AC}, at the beginning of which we list what remains to be proved after the reductions.

In \S\ref{ss:final} we comment on a possible way to relate $(\RT_{\min}, \KL)$ and $(\Phi,\Psi)$ directly.

\subsection{Other results} Along the way we also prove some results that are of independent interest: in Theorem \ref{th:min length} we give a formula for the minimal length in an elliptic conjugacy class of $W$ in terms of root data; in Corollary \ref{c:dual order} we show that the partial order on $\un W$ defined by Spaltenstein depends only on $W$ and not on the root system of $G$; in Corollary \ref{c:Phi order} we show that  $\RT_{\min}$ and $\KL$ respect partial orders on $\un\cN$ and $\un W$.



\section{A formula for $\d_{[w]}$}
In this section we  relate $\d_{[w]}$ and the minimal length for elements in $[w]$. This result will be used in the proof of Theorem \ref{th:WtoN}.

\subsection{Filtration on the root system} Suppose $w\in W$ is of order $m\ge1$. For each $i|m$, let $\frt^{*}[i]\subset \frt^{*}$ be the sum of eigenspaces of $w$ with eigenvalues primitive $i$-th roots of unity. For example, $\frt^{*}[1]=(\frt^{*})^{w}$. 

Let $R$ be the set of roots of $G$ with respect to $T$. For $i|m$  let
\begin{equation*}
R_{\le i}= R\cap(\oplus_{1\le i'\le i, i'|m}\frt^{*}[i']).
\end{equation*}
This is a saturated sub root system of $R$.   Set $R_{\le 0}=\vn$ and let $R_{i}:=R_{\le i}-R_{\le i-1}$. Note our $R_{i}$ is different from the same-named subset of $R$ in \cite[3.4]{GKM}.

\begin{lemma} Suppose $w\in W$ has order $m$. Then for any $\g\in (L^{\hs}\frg)^{sh}_{[w]}$, we have
\begin{equation}\label{Dmin}
\val(\D(\g))=\sum_{i|m}\frac{|R_{i}|}{i}.
\end{equation}
\end{lemma}
\begin{proof}
Let $LT_{w}$ be a loop torus of type $[w]$ and assume $\g\in (L^{\hs}\frt_{w})^{sh}$. Following \cite[4.3]{GKM}, one can describe $L\frt_{w}$ as  the fixed point subspace of $\frt\lr{t^{1/m}}$ under the automorphisms $a(t^{1/m})\mapsto wa(\z^{-1}t^{1/m})$ where $\z$ is a primitive $m$-th root of unity. For $j\in\ZZ/m\ZZ$, let $\frt_{j}\subset \frt$ be the eigenspace of $w$ with eigenvalue $\z^{j}$, then $L\frt_{w}=\sum_{n\in\ZZ_{\ge0}}t^{n/m}\frt_{\un n}$ (here $\un n$ is $n$ mod $m$). Write $\g=\sum_{n}t^{n/m}\g_{n}$ with $\g_{n}\in \frt_{\un n}$. Then $\g\in (L^{\hs}\frt_{w})^{sh}$ if and only if $\g_{0}=0$ and for each root $\a\in R$, $\val(\a(\g))$ achieves the minimum among $\g\in L^{\hs}\frt_{w}$. Note that these minima can be achieved simultaneously because the minimum for each  $\val(\a(\g))$  is achieved on an open dense subset of $L^{\hs}\frt_{w}$. In other words, if $j_{\a}\in\NN$ is the smallest $j\in \{1,2,\cdots, m\}$ such that $\a|_{\frt_{\un j}}\ne0$, then $\a(\g_{j})\ne0$. Let $\frt_{\QQ}$ be the rational span of the coroots. Then the decomposition  $\frt=\op_{i|m}\frt[i]$ descends to $\frt_{\QQ}=\op_{i|m}\frt_{\QQ}[i]$, and $\frt_{\QQ}[i]_{\CC}\cong \frt_{j}$ (via projection) for various $j\in\ZZ/m$ such that $\z^{j}$ has order $i$ (i.e., $\gcd(j,m)=m/i$).  Since roots are in $\frt_{\QQ}^{*}$, for any $i|m$, $\a|_{\frt[i]}\ne0$ if and only if $\a|_{\frt_{\un {m/i}}}\ne0$. Therefore,  $\a\in R_{i}$ iff $\a|_{\frt[i']}=0$ for all $i'>i$ and $\a|_{\frt[i]}\ne0$, iff $j_{\a}=m/i$. This implies $\val(\D(\g))=\sum_{\a\in R}\val(\a(\g))=\sum_{i|m}|R_{i}|\cdot \frac{m/i}{m}=\sum_{i|m}\frac{|R_{i}|}{i}$.
\end{proof}

For $[w]\in \un W$, let $\ell_{\min}([w])$ be the minimum of $\ell(w')$ where $w'$ runs over the conjugacy class $[w]$. Here $\ell(-)$ is the length function on $W$ with respect to any set $S$ of simple reflections (the resulting function $\ell_{\min}$ is independent of the choice of $S$).

\begin{theorem}\label{th:min length} Let $[w]\in \un W_{ell}$, then we have 
\begin{equation*}
\d_{[w]}=\frac{\ell_{\min}([w])-r}{2}
\end{equation*}
where $r$ is the semisimple rank of $G$.
\end{theorem}
\begin{proof} The argument here is an easy modification of Springer's argument  in \cite[Prop 4.10]{Spr-reg}. 

Take any $\g\in(L^{\hs}\frt_{[w]})^{sh}$. We need to show that $\d_{\g}=\frac{\ell_{\min}([w])-r}{2}$. Comparing with \eqref{dim Sp}, we need to show that $\ell_{\min}(w)=\val(\D(\g))$. Using \eqref{Dmin}, we reduce to showing that
\begin{equation*}
\ell_{\min}([w])=\sum_{i|m}\frac{|R_{i}|}{i}.
\end{equation*}
We easily reduce to the case $G$ is almost simple. Since $w$ is elliptic, $\frt^{*}[1]=0$, hence $R_{\le 1}=\vn$. For a choice of a basis for $R$, or equivalently a choice of positive roots $R^{+}$. Let $R^{+}_{w}=\{\a\in R^{+}|w\a\notin  R^{+}\}$. Then $\ell(w)=| R^{+}_{w}|$.

Let $\fra^{*}\subset \frt^{*}$ be the $\RR$-span of the roots. Let $\fra^{*}[i]=\fra^{*}\cap\frt^{*}[i]$, $\fra^{*}[\le i]:=\op_{i'\le i, i'|m}\fra^{*}[i']$. We have $R_{\le i}=R^{+}_{\le i}\sqcup (-R^{+}_{\le i})$ where $R^{+}_{\le i}=R_{\le i}\cap  R^{+}$. Any such decomposition is cut out by a generic real half space of $\fra^{*}[i]$: there exists a real half subspace $\fra^{*}[i]^{+}\subset\fra^{*}[i]$ such that $\a\in R_{\le i}-R_{\le i-1}$ belongs to $R^{+}_{\le i}$ if and only if the image of $\a$ in $\fra^{*}[i]$ lies in $\frt^{*}[i]^{+}$. Conversely, any generic choice of such a real half space, one for each $\fra^{*}[i]$, gives a choice of positive roots $ R^{+}$.  

Let $i|m, i\ge2$. Consider a generic real half space  $\fra^{*}[i]^{+}$. For any $\a\in R_{i}$, the roots $\{\a,w\a,\cdots, w^{i-1}\a\}$ cannot be all positive or all negative, for their sum has image zero in $\fra^{*}[i]$. Let $\{\a, w\a,\cdots, w^{n-1}\a\}$ be the $\j{w}$-orbit of $\a$ , so that $i|n|m$. Then for any $\a\in R_{i}$, among  the orbit $\{\a, w\a,\cdots, w^{n-1}\a\}$, at least $n/i$ roots $\b$ satisfy that $\b>0$ but $w\b<0$ (there is at least one such $\b=w^{j}\a$ for $0\le j\le i-1$, and then $w^{i}\b, w^{2i}\b,\cdots$ also satisfy the same condition). Therefore, $|R_{i}\cap  R^{+}_{w}|\ge |R_{i}|/i$.  This proves the inequality $\min\{\ell(w'), w'\in [w]\}=\min\{| R_{w'}^{+}|, w'\in [w]\}\ge\sum_{i|m, i\ge 2}|R_{i}|/i=\sum_{i|m}|R_{i}|/i$. 

To achieve the equality, we choose the real half plane $\fra^{*}[i]^{+}$ to be defined by $\textup{Re}\j{v,-}>0$ for some $v\in \frt(\z)$, where $\frt(\z)\subset \frt$ is the eigenspace of $w$ on $\frt$ with eigenvalue a primitive $i$-th root of unity $\z$. We need to choose $v$ generically so that $\textup{Re}\j{v,\pi(\a)}\ge0$ for any $\a\in R_{i}$ with image $\pi(\a)\in \fra^{*}[i]$. This is possible if $\j{\frt(\z), \a}\ne0$. We claim that $\j{\frt(\z), \a}\ne0$ holds for any $\a\in R_{i}$. If not, by the rationality of $\a$, $\j{\frt(\z'),\a}$ would be identically zero for any primitive $i$th root of unity $\z'$, hence $\pi(\a)$ is zero as a linear function on $\frt[i]$, i.e., $\pi(\a)=0$, which is a contradiction.  
\end{proof}

\begin{cor}\label{c:d ell} Suppose $[w]\in \un W_{ell}$ and $\Phi([w])=\cO$ (here $\Phi$ is Lusztig's map \eqref{Phi}). Then $\d_{[w]}=d_{\cO}$.
\end{cor}
\begin{proof}
By Theorem \ref{th:min length}, we have $\d_{[w]}=(\ell_{\min}([w])-r)/2$. By \cite[Theorem 0.7]{Lfromto}, $\ell_{\min}([w])=\dim Z_{G^{\ad}}(e)$ for any $e\in \cO$. Combining these equalities we get $\d_{[w]}=(\dim Z_{G^{\ad}}(e)-r)/2$, which is equal to $d_{\cO}$ by Steinberg's formula \cite[Theorem 4.6]{St}.
\end{proof}

\section{First reductions}
\subsection{Notations involving a Levi subgroup}\label{ss:Levi} Let $M\subset G$ be a Levi subgroup. It induces a map
\begin{equation*}
\io=\io_{M,G}: \un \cN_{M}\to \un\cN
\end{equation*}
sending a nilpotent orbit $\cO'$ of $M$ to the nilpotent orbit $\cO$ in $G$ containing $\cO'$. This map is not injective in general. This is {\em not} the Lusztig-Spaltenstein induction.

If $M$ is a Levi subgroup of $G$ containing $T$, then the Weyl group $W_{M}=W(M,T)$ is a parabolic subgroup of $W$. This induces a map
\begin{equation*}
\mj=\mj_{M,G}: \un W_{M}\to \un W. 
\end{equation*}
Again $\mj$ is not injective in general. 

If $w\in W_{M}$, denote by $[w]_{M}$ the conjugacy class of $w$ in $W_{M}$; denote by $[w]_{G}$ or simply $[w]$ its conjugacy class in $W$.

In general, we use superscript or subscript $M$ to emphasize objects $M$-counterparts of objects defined for $G$, e.g.,  $\RT^{M}_{\min}(-)$.

\begin{lemma}\label{l:RT Levi} Let $M\subset G$ be a Levi subgroup of $G$, and $\g\in (L\fm)\cap L^{\hs}\frg$ (where $\fm=\Lie M$). Then
\begin{equation*}
\RT_{\min}(\g)\subset \io_{M,G}(\RT^{M}_{\min}(\g)).
\end{equation*}
\end{lemma}
\begin{proof}
Let $A_{M}$ be the connected center of $M$, then $\Gr_{M}\to \Gr^{A_{M}}$ is an isomorphism on reduced structures. Then $\Gr_{M,\g}=(\Gr_{\g})^{A_{M}}$ as reduced closed subschemes of $\Gr_{\g}$. We have a commutative diagram
\begin{equation*}
\xymatrix{   \Gr_{M,\g}\ar[r]^{\ev^{M}_{\g}}\ar@{^{(}->}[d] & [\cN_{M}/M]\ar[d]        \\
\Gr_{\g}\ar[r]^{\ev^{G}_{\g}} & [\cN/G]}
\end{equation*} 
From this we see that $\io(\RT^{M}(\g))\subset \RT(\g)$. Now choose a generic cocharacter $\l: \Gm\to A_{M}$ such that $(\Gr_{\g})^{\l(\Gm)}=(\Gr_{\g})^{A_{M}}$. Every point $x\in \Gr_{\g}$ contracts to a point $x_{0}\in \Gr_{M,\g}$ under the left translation action by $\l(\Gm)$. Now $\io(\ev^{M}_{\g}(x_{0}))=\ev^{G}_{\g}(x_{0})$ lies in the  closure of $\ev^{G}_{\g}(x)$. If $\ev^{G}_{\g}(x)\in \RT_{\min}(\g)$, we see that $\io(\ev^{M}_{\g}(x_{0}))=\ev^{G}_{\g}(x)$ is also in $\RT_{\min}(\g)$, which means every element in $\RT_{\min}(\g)$ is contained in $\io(\RT^{M}(\g))$, hence in $\io(\RT^{M}_{\min}(\g))$.
\end{proof}

\begin{remark}\label{r:nonLevi red} If $M$ is only assumed to be a reductive subgroup of $G$ containing $T$ but not necessarily a Levi subgroup, then the conclusion of Lemma \ref{l:RT Levi} does not necessarily hold. For example, $G=\SO_{8}$ and $M=\SO_{4}\times \SO_{4}$. We take $\g=(\g_{1},\g_{2})$ where $\g_{1},\g_{2}$ are shallow of type $-1$ in $L^{\hs}\so_{4}$, and $\g$ is shallow of type $-1$ in $L^{\hs}\so_{8}$ (we use $-1$ to denote the longest element in the Weyl group in these cases). Then $\RT^{M}_{\min}(\g)$ is a singleton with Jordan type $3311$ while $\RT^{G}_{\min}(\g)$ is a singleton with Jordan type $3221$. 
\end{remark}

\begin{cor}\label{c:reduce to Levi} Let $M$ be a Levi subgroup of $G$  containing $T$. 
\begin{enumerate}
\item If Conjectures \ref{c:min red type} holds for $M$, then Conjectures \ref{c:min red type} holds for any $\g\in  L^{\hs}\frg$ that can be conjugated to an element of $L\fm$, and $\RT_{\min}(\g)=\io_{M,G}(\RT^{M}_{\min}(\g))$.
\item Let $w\in W_{M}$. If Theorem \ref{th:WtoN}\eqref{Uw} and \ref{th:PhiPsi}\eqref{Phi} hold for $(M,[w]_{M})$, then they hold for $(G,[w])$, and $\RT_{\min}([w])=\io_{M,G}(\RT_{\min}^{M}([w]_{M}))$. 
\end{enumerate}
\end{cor}
\begin{proof}
(1) follows directly from Lemma \ref{l:RT Levi} since $\RT_{\min}(\g)\subset \io(\RT^{M}_{\min}(\g))$. 

(2) Let $\g\in(L^{\hs}\frg)_{[w]}^{sh}$. Let $LT_{w}\subset LM$ be a loop torus of type $[w]_{M}$. Up to $LG$-conjugation we may assume $\g\in (L^{\hs}\frt_{w})^{sh}$. Since Theorem \ref{th:WtoN}\eqref{Uw} holds for $(M,[w]_{M})$, we have $\RT^{M}_{\min}(\g)=\{\cO'\}$ for some $\cO'\in \un \cN_{M}$ depending only on $[w]_{M}$. Let $\cO=\io_{M,G}(\cO')$. By Lemma \ref{l:RT Levi}, we have $\RT_{\min}^{G}(\g)\subset \io(\RT^{M}_{\min}(\g))=\{\cO\}$. Hence $\RT_{\min}(\g)=\{\cO\}$ which only depends on $[w]_{M}$. 

We still need to argue that $\cO$ depends only on $[w]$, and not just on $[w]_{M}$. Suppose $v\in W_{M}$ is in the same $W$-conjugacy class of $w$. Then we may find $x\in N_{W_{M}}(W)$ such that $v=xwx^{-1}$. Lifting  $x$ to $\dot{x}\in N_{G}(M)$, and applying the automorphism $\Ad(\dot{x})$ of $M$, we see that Theorem \ref{th:WtoN}\eqref{Uw} holds for $(M,[v]_{M})$, and $\RT^{M}_{\min}([v]_{M})=\{\Ad(\dot x)(\cO')\}$. Since $\Ad(\dot x)(\cO')$ and $\cO'$ both lie in $\cO$, we have arrived at the same nilpotent orbit in $G$ starting from $w$ and $v$. Therefore $\RT_{\min}(\g)$ depends only on $[w]$.

Now Theorem \ref{th:PhiPsi}\eqref{Phi} also hold for $(G,[w])$ because Lusztig's map satisfies $\Phi_{M}([w]_{M})=\Phi_{G}([w])$ by \cite[1.1]{Lfromto}. 
\end{proof}

\begin{lemma}\label{l:isog}
\begin{enumerate}
\item Let  $G$ and $G'$ be isogenous reductive groups, then for any $\g\in  L^{\hs}\frg$, $\RT^{G}(\g)=\RT^{G'}(\g)$. 
\item Let  $G$ and $G'$ be isogenous reductive groups, then each of Theorem \ref{th:WtoN}\eqref{Uw}, Conjectures \ref{c:min red type} and \ref{c:trans} holds for  $G$ if and only if it holds for $G'$.
\item For each of the statements in Theorems \ref{th:WtoN} and \ref{th:PhiPsi}, Conjectures \ref{c:min red type} and \ref{c:trans}, it holds for $G$ if and only if it holds for every simple factor of $G$ (well-defined up to isogeny).
\end{enumerate}
\end{lemma}
\begin{proof}
(1) Identify the Lie algebras $\frg=\frg'$, and take $\g\in L^{\hs}\frg$. Let $\Gr^{\c}_{G}$ and $\Gr^{\c}_{G'}$ be the neutral components of $\Gr_{G}$ and $\Gr_{G'}$. Then by \cite[Proposition 6.6]{PR} $\Gr^{\c}_{G}=\Gr^{\c}_{G'}$, and
\begin{equation}\label{Grgc}
\Gr^{\c}_{G,\g}=\Gr^{\c}_{G',\g},
\end{equation}
where $(-)^{\c}$ denotes intersection with $\Gr^{\c}_{G}$ or $\Gr^{\c}_{G'}$. Moreover, $(LG)_{\g}$ acts transitively on the set of connected components of $\Gr_{G}$, hence
\begin{eqnarray}
\label{GrG trans}\Gr_{G,\g}=(LG)_{\g}\cdot \Gr^{\c}_{G,\g},\\
\label{GrG' trans}\Gr_{G',\g}=(LG')_{\g}\cdot \Gr^{\c}_{G',\g}.
\end{eqnarray}
Since $\ev_{G,\g}$ and $\ev_{G',\g}$ are invariant under $(LG)_{\g}$ and $(LG')_{\g}$ respectively, we get that $\RT^{G}(\g)=\RT^{G'}(\g)$.

(2) The equivalences for $G$ and $G'$ of Theorems \ref{th:WtoN}\eqref{Uw} and Conjecture \ref{c:min red type} follow from (1) immediately. For Conjectures \ref{c:trans}, we may reduce to the case where $G\to G'$ is a finite isogeny. Let $\g\in L^{\hs}\frg$ be shallow of type $[w]=\KL(\cO)$. From \eqref{Grgc}, \eqref{GrG trans} and \eqref{GrG' trans} we get
\begin{eqnarray}
\label{GG'O}\Gr^{\c}_{G,\g,\cO}=\Gr^{\c}_{G',\g,\cO},\\
\label{GO trans}\Gr_{G,\g,\cO}=(LG)_{\g}\cdot \Gr^{\c}_{G,\g,\cO},\\
\label{G'O trans}\Gr_{G',\g,\cO}=(LG')_{\g}\cdot \Gr^{\c}_{G',\g,\cO}.
\end{eqnarray}
Let $(LG)^{\c}$ be the neutral component of $LG$.  Now $(LG)_{\g}\cap (LG)^{\c}$ acts on $\Gr^{\c}_{G,\g,\cO}$ via the surjection $(LG)_{\g}\cap (LG)^{\c}\surj (LG')_{\g}\cap (LG')^{\c}$, therefore $\Gr^{\c}_{G,\g,\cO}$ is transitive under $(LG)_{\g}\cap (LG)^{\c}$ if and only if the same is true for $G'$. By \eqref{GO trans}, $\Gr^{\c}_{G,\g,\cO}$ is transitive under $(LG)_{\g}\cap (LG)^{\c}$ if and only if $\Gr_{G,\g,\cO}$ is transitive under $(LG)_{\g}$. By By \eqref{G'O trans}, the same is true for $G'$. Combining these, we conclude that $\Gr_{G,\g,\cO}$ is transitive under $(LG)_{\g}$ if and only if $\Gr_{G',\g,\cO}$ is transitive under $(LG')_{\g}$.

(3) $G$ is isogenous to a product $A\times \prod_{i}G_{i}$ where $G_{i}$ are simple and $A$ is a torus. If one of the statements hold for all $G_{i}$, it holds for the product $A\times \prod_{i}G_{i}$, hence for $G$ by (2).
\end{proof}

\section{The subsets $(L^{\hs}\frc)_{[w]}$}\label{s:Lcw}

\subsection{Fp constructible subsets}

For a scheme $X$ of finite type over $\CC$ and $n\ge0$, let $L^{+}_{n}X$ be its truncated arc space representing the functor $R\mapsto X(R\tl{t}/t^{n+1})$. Then the arc space $L^{+}X$ is the limit $\varprojlim_{n}L^{+}_{n}X$ with natural projections $\pi_{n}: L^{+}X\to L^{+}_{n}X$.

A subset $Z\subset (L^{+}X)(\CC)=X(\CC\tl{t})$ is called fp constructible (where fp stands for ``finite presentation'') if there is $n\in\NN$ and a Zariski constructible subset $Z_{n}\subset (L^{+}_{n}X)(\CC)=X(\CC\tl{t}/(t^{n+1}))$ such that $Z=\pi_{n}^{-1}(Z_{n})$.  Similarly there is the notion of fp open, fp closed, fp locally closed subsets of $L^{+}X$. 

If $Z\subset (L^{+}X)(\CC)$ is fp constructible, say $Z=\pi_{n}^{-1}(Z_{n})$ for a constructible $Z\subset (L^{+}_{n}X)(\CC)$, then the closure $\ov Z$ of $Z$ (inside $L^{+}X$) is defined as the fp closed subset $\pi^{-1}_{n}(\ov Z_{n})\subset (L^{+}X)(\CC)$, where $\ov Z_{n}\subset (L^{+}_{n}X)(\CC)$ is the Zariski closure of $Z_{n}$ in $(L^{+}_{n}X)(\CC)$. This is independent of the choices of $n$ and $Z_{n}$.

For a  fp constructible $Z\subset (L^{+}X)(\CC)$, say $Z=\pi_{n}^{-1}(Z_{n})$ for a constructible $Z\subset (L^{+}_{n}X)(\CC)$, we define $\codim_{L^{+}X}(Z)=\codim_{L^{+}_{n}X}(Z_{n})$. Again this is independent of the choices of $n$ and $Z_{n}$.

In the sequel, we simply write $(L^{+}\frc)(\CC)$ as $L^{+}\frc$.  If $Z$ is a fp constructible subset of $L^{+}\frc$ that is contained in $L^{++}\frc$ (or in $L^{\hs}\frc$), then $\ov Z\subset L^{++}\frc$, and  $\codim_{L^{++}\frc}(Z)$ is defined as $\codim_{L^{+}\frc}(Z)-\dim \frc=\codim_{L^{+}\frc}(Z)-\dim \frt$.

Recall the subsets $(L^{\hs}\frc)^{sh}_{[w]}\subset (L^{\hs}\frc)_{[w]}$ of $L^{\hs}\frc$ from \S\ref{ss:intro c}.

\begin{lemma}\label{l:cw}
\begin{enumerate}
\item $(L^{\hs}\frc)^{sh}_{[w]}$ is an irreducible fp locally closed subset of $L^{+}\frc$.
\item $\codim_{L^{++}\frc}(L^{\hs}\frc)^{sh}_{[w]}=\d_{[w]}$.
\item $(L^{\hs}\frc)_{[w]}$ is a countable union of fp locally closed subsets of $L^{+}\frc$, and $(L^{\hs}\frc)^{sh}_{[w]}$ is dense in it.
\end{enumerate}
\end{lemma}
\begin{proof}
(1) Since $(L^{\hs}\frc)^{sh}_{[w]}$ is a special case of a root valuation strata, the statement follows from \cite[7.2.8(d)]{BKV}. Moreover, by \cite{BKV}, $(L^{\hs}\frt_{w})^{sh}\to (L^{\hs}\frc)^{sh}_{[w]}$ is surjective and finitely presented (indeed finite \'etale). Since $(L^{\hs}\frt_{w})^{sh}$ is irreducible (being open in $L^{++}\frt_{w}$), so is $(L^{\hs}\frc)^{sh}_{[w]}$.

(2) By \cite[Theorem 8.2.2]{GKM} applied to the root valuation stratum $(L^{\hs}\frc)^{sh}_{[w]}$, we have
\begin{equation}
\codim_{L^{+}\frc}(L^{\hs}\frc)^{sh}_{[w]}=\codim_{L^{+}\frt_{w}}(L^{\hs}\frt_{w})^{sh}+\d_{[w]}+(\dim\frt-\dim\frt^{w}).
\end{equation}
Using that $\codim_{L^{+}\frt_{w}}(L^{\hs}\frt_{w})^{sh}=\dim \frt^{w}$ we get $\codim_{L^{+}\frc}(L^{\hs}\frc)^{sh}_{[w]}=\d_{[w]}+\dim\frt$. Hence  $\codim_{L^{++}\frc} (L^{\hs}\frc)^{sh}_{[w]}=\codim_{L^{+}\frc}(L^{\hs}\frc)^{sh}_{[w]}-\dim\frt=\d_{[w]}$.

(3) $(L^{\hs}\frc)_{[w]}$ is the union of root valuation strata, each of which is fp locally closed. Since $(L^{\hs}\frt_{w})^{sh}$ is dense in $L^{\hs}\frt_{w}$, its image $(L^{\hs}\frc)^{sh}_{[w]}$ in $L^{+}\frc$ is dense in $(L^{\hs}\frc)_{[w]}$.
\end{proof}

\subsection{Partial order on $\un W$}\label{ss:order W} Spaltenstein \cite{Spa-order} defined a partial order on $\un W$ as follows: for $[w_{1}],[w_{2}]\in \un W$,   we denote $[w_{2}]\preceq[w_{1}]$ if $(L^{\hs}\frc)_{[w_{2}]}\subset \ov{(L^{\hs}\frc)_{[w_{1}]}}$. By Lemma \ref{l:cw}, $[w_{2}]\preceq[w_{1}]$ if and only if $(L^{\hs}\frc)^{sh}_{[w_{2}]}\subset \ov{(L^{\hs}\frc)_{[w_{1}]}}=\ov{(L^{\hs}\frc)^{sh}_{[w_{1}]}}$.

A priori the partial order $\preceq$ on $\un W$ depends not only on $W$ but also on $\frg$. However we shall see in Corollary \ref{c:dual order} that $\preceq$ only depends on $W$.

\begin{lemma}\label{l:dw ineq}
If $[w_{2}]\prec [w_{1}]$, then $\d_{[w_{2}]}>\d_{[w_{1}]}$.
\end{lemma}
\begin{proof}
By Lemma \ref{l:cw}, $\codim_{L^{\hs}\frc}(L^{\hs}\frc)^{sh}_{[w_{i}]}=\d_{[w_{i}]}$. Since $[w_{2}]\prec [w_{1}]$, $(L^{+}\frc)^{sh}_{[w_{2}]}\subset \ov{(L^{+}\frc)^{sh}_{[w_{1}]}}\bs (L^{+}\frc)^{sh}_{[w_{1}]}$, the latter having codimension $>\d_{[w_{1}]}$ in $L^{\hs}\frc$. Thus we have $\d_{[w_{2}]}> \d_{[w_{1}]}$. 
\end{proof}

\begin{lemma}[Generalization of Spaltenstein {\cite[10.2]{Spa-order}}]\label{l:order pres}
Let $H$ be a connected reductive subgroup with maximal torus $T_{H}$. Let $T_{H}\to T$ be an isogeny (surjective with finite kernel) under which the Weyl group $W_{H}=W(H,T_{H})$ is a subgroup of $W$. Then induced the map $\un W_{H}\to \un W$ is order-preserving under Spaltenstein's partial orders.
\end{lemma}
\begin{proof}
For each $w\in W_{H}$, we choose a loop torus $LT_{w}\subset LH$ of type $[w]_{H}$. We use superscripts $\hs_{H}$ and $sh_{H}$ to emphasize ``topologically nilpotent and generically regular semisimple'' and ``shallow'' when viewed as elements in $L\frh$, while $\hs$ is used to mean the same thing for $L\frg$. For example we have $(L^{\hs_{H}}\frc_{H})_{[w]_{H}}$ and  $(L^{\hs_{H}}\frc'_{H})_{[w]_{H}}^{sh_{H}}$.


Since $T_{H}\to T$ is an isogeny and $W_{H}\subset W$, we have a natural maps $\frt\to \frc_{H}\to \frc$. The loop torus $LT_{w}$ is isogenous to a loop torus of the same type in $LG$, and we may identify $L\frt_{w}$ with a subalgebra of $L\frg$. We have natural maps
\begin{equation*}
L^{+}\frt_{w}\to L^{+}\frc_{H}\to L^{+}\frc.
\end{equation*}
We have an open dense embedding $i_{w}: (L^{\hs}\frt_{w})^{sh}\subset (L^{\hs_{H}}\frt_{w})^{sh_{H}}$ since roots of $H$ are proportional to a subset of roots of $G$. Let $(L^{\hs}\frc_{H})^{sh}_{[w]_{H}}$ be the image of $(L^{\hs}\frt_{w})^{sh}$ under $L^{+}\frt_{w}\to L^{+}\frc_{H}$. Consider the commutative diagram
\begin{equation*}
\xymatrix{  (L^{\hs}\frt_{w})^{sh}\ar@{^{(}->}[d]^{i_{w}}\ar@{->>}[r] & (L^{\hs}\frc_{H})^{sh}_{[w]_{H}}\ar@{^{(}->}[d]^{j_{[w]_{H}}} \ar@{->>}[r] & (L^{\hs}\frc)^{sh}_{[w]}\\
(L^{\hs_{H}}\frt_{w})^{sh_{H}}\ar@{->>}[r] & (L^{\hs_{H}}\frc_{H})^{sh_{H}}_{[w]_{H}} }
\end{equation*}
The horizontal maps are surjective. Since $i_{w}$ has dense image, hence so does $j_{[w]_{H}}$. 

Now if $w,w'\in W_{H}$ and $[w']_{H}\preceq [w]_{H}$, then $(L^{\hs_{H}}\frc_{H})^{sh_{H}}_{[w']_{H}}\subset \ov{(L^{\hs_{H}}\frc_{H})^{sh_{H}}_{[w]_{H}}}$ (closure taken in $L^{+}\frc_{H}$).  This implies $(L^{\hs}\frc_{H})^{sh}_{[w']_{H}}\subset \ov{(L^{\hs}\frc_{H})^{sh}_{[w]_{H}}}$ (closure taken again in $L^{+}\frc_{H}$) since $j_{[w]_{H}}$ has dense image. Taking image in $L^{\hs}\frc$ we conclude that $(L^{\hs}\frc)^{sh}_{[w']}\subset \ov{(L^{\hs}\frc)^{sh}_{[w]}}$, i.e., $[w']\preceq [w]$.
\end{proof}

\begin{cor}\label{c:dual order} Suppose $G$ and $G'$ are reductive groups with isogenous root systems (hence the same Weyl group $W$), then the partial orders on $W$ defined using $G$ and $G'$ are the same. In particular, the partial orders on $W$ defined using $G$ and its Langlands dual $G^{\vee}$ are the same.
\end{cor}

\section{The subsets  $(L^{\hs}\mathfrak{c})_{\overline{\mathcal{O}}}$}\label{s:c} 

\subsection{Subsets of $L^{\hs}\frc$}\label{ss:L hs c}

Recall the discriminant function $\D: \frg\to\AA^{1}$ in \S\ref{ss:d} factors through $\frc$, which we still denote by $\D:\frc\to \AA^{1}$. It induces a map of arc spaces $L^{+}\D: L^{+}\frc\to L^{+}\AA^{1}$, the latter is an infinite-dimensional affine space. By definition,  $L^{\hs}\frc=(L^{++}\frc)(\CC)\bs (L^{+}\D)^{-1}(0)$. From this we see that $L^{\hs}\frc$ is the set of $\CC$-points of the open subscheme $L^{++}\frc\bs (L^{+}\D)^{-1}(0)$ of $L^{++}\frc$.

Let  $(L^{+}\frc)_{\le n}$ be the subset in $L^{+}\frc$ where the valuation of the discriminant function $\D$ is $\le n$, i.e., $(L^{+}\frc)_{\le n}=L^{+}\frc\bs (L^{+}\D)^{-1}(t^{n+1}\CC\tl{t})$. Denote $(L^{\hs}\frc)_{\le n}=L^{\hs}\frc\cap (L^{+}\frc)_{\le n}$. Then $L^{\hs}\frc=\cup_{n\ge0}(L^{\hs}\frc)_{\le n}$. 
Clearly $(L^{+}\frc)_{\le n}$ is a fp open subset of $L^{+}\frc$, and $(L^{\hs}\frc)_{\le n}$ is fp open in $L^{++}\frc$.

A subset $Z\subset L^{\hs}\frc$ is called fp constructible (resp. closed, locally closed) if $Z\cap (L^{\hs}\frc)_{\le n}$ is fp closed (resp. constructible, locally closed) as a subset of $L^{+}\frc$ for any $n\ge0$. For a fp constructible subset $Z\subset L^{\hs}\frc$, we define 
\begin{equation*}
\codim_{L^{\hs}\frc}(Z):=\min\{\codim_{(L^{\hs}\frc)_{\le n}}(Z\cap (L^{\hs}\frc)_{\le n}); n\in \NN\}.
\end{equation*}

Let $(L^{\hs}\frg)_{\le n}\subset (L^{+}\frg)_{\le n}\subset L^{+}\frg$ be the preimages of $(L^{\hs}\frc)_{\le n}$ and $(L^{+}\frc)_{\le n}$ in $L^{+}\frg$ under $L^{+}\chi: L^{+}\frg\to L^{+}\frc$. 

We need the following technical lemma.
\begin{lemma}\label{l:cons} For any fp constructible set $Z\subset (L^{+}\frg)_{\le n}$, $(L^{+}\chi)(Z)\subset (L^{+}\frc)_{\le n}$ is also fp constructible in $L^{+}\frc$.\end{lemma}
\begin{proof} Below we write $L^{+}\chi$ simply as $\chi$. Choose  a closed embedding $i:\frg\incl \frc\times \AA^{N}$. To this embedding Elkik \cite[0.2]{Elkik} has assigned an ideal $H\subset \cO(\frc\times\AA^{N})$, or equivalently a closed subscheme $V_{H}\subset \frg\times \AA^{N}$ with the property that $V_{H}\cap i(\frg)$ is the singular locus of $\chi$. Since $\frg-\{\D=0\}$  is contained in the smooth locus of $\chi$, we see that $V_{H}\cap i(\frg)\subset \{\D^{m}=0\}$ scheme-theoretically for some $m\ge1$. For any $b\in (L^{+}\frg)_{\le n}$ viewed as a $\CC\tl{t}$-point of $\frg$, the pullback $(i\circ b)^{-1}V_{H}$, as a subscheme of $\Spec\CC\tl{t}$, has the form $\Spec \CC\tl{t}/t^{h}$ for $h\in mn$. In other words, $b\in (L^{+}\frg)_{\le n}$ has bounded conductor with respect to $\chi$. Applying \cite[Theorem 1]{Elkik} to $\chi$, we see that there exists $(d_{0}, r)$ (depending on $m,n$) such that for any $d>d_{0}$, any $a\in (L^{+}\frc)_{\le n}$ and $b\in (L^{+}\frg)_{\le n}$ such that $\chi(b)\equiv a\mod t^{d+r}$, there exists $b'\in L^{+}\frg$ such that $\chi(b')=a$ and $b\equiv b'\mod t^{d}$. We may choose $d_{0}$ large enough so that $b\equiv b'\mod t^{d}$ implies $b'\in(L^{+}\frg)_{\le n}$.

Now let $Z\subset (L^{+}\frg)_{\le n}$ be fp constructible, i.e., there exists $k$ such that Z is the preimage of a constructible subset $Z_{k}\subset L^{+}_{k}\frg=\frg(\CC\tl{t}/t^{k})$.  We may assume $k>d_{0}-r$. Then let $Z_{k+r}$ be the image of $Z$ in $L^{+}_{k+r}\frg$, and let $Y_{k+r}=\chi(Z_{k+r})\subset L^{+}_{k+r}\frc$ which is constructible. We claim that $\chi(Z)=\pi_{k+r}^{-1}(Y_{k+r})$ where $\pi_{k+r}: L^{+}\frc\to L^{+}_{k+r}\frc$ is the projection.  This  implies $\chi(Z)$ is fp constructible. Now $\chi(Z)$ is clearly contained in $\pi_{k+r}^{-1}(Y_{k+r})$; conversely if $a\in \pi_{k+r}^{-1}(Y_{k+r})$ viewed as a $\CC\tl{t}$-point of $\frc$, then $a\mod t^{k+r}\in Y_{k+r}$, hence $a\equiv \chi(b)\mod t^{k+r}$ for some $b\in Z_{k+r}$. By the discussion in the previous paragraph, there exists $b'\in (L^{+}\frg)_{\le n}$ such that $b'\equiv b\mod t^{k}$ and $\chi(b)=a$. Now $b'\equiv b\mod t^{k}$ implies $b'\in Z$, therefore $a\in\chi(Z)$. This proves the lemma.
\end{proof}

For $\cO\in \un\cN$, let $(L^{\hs}\frc)_{\ov\cO}\subset L^{\hs}\frc$ be the image of $(\ov\cO+tL^{+}\frg)\cap L^{\hs}\frg$ under  $L^{+}\chi: L^{+}\frg\to L^{+}\frc$.

\begin{lemma}\label{l:cO}
$(L^{\hs}\frc)_{\ov\cO}$ is a fp closed subset of $L^{\hs}\frc$. 
\end{lemma}
\begin{proof}
We need to show that $(L^{\hs}\frc)_{\ov\cO}\cap (L^{\hs}\frc)_{\le n}$ is fp closed in $(L^{\hs}\frc)_{\le n}$. 

By Lemma \ref{l:cons}, $(L^{\hs}\frc)_{\ov\cO}\cap(L^{\hs}\frc)_{\le n}$ is a fp constructible subset of $L^{+}\frc$. Therefore,   $(L^{\hs}\frc)_{\ov\cO}\cap(L^{\hs}\frc)_{\le n}$ is the preimage of a constructible subset $Y\subset L^{+}_{N}\frc$ for some $N\in \NN$. By enlarging $N$ we may also assume $(L^{\hs}\frc)_{\le n}$ is the preimage of a locally closed subset $(L^{\hs}_{N}\frc)_{\le n}\subset L_{N}^{+}\frc$.  Then $Y\subset (L^{\hs}_{N}\frc)_{\le n}$. Let $Y'$ be the closure of $Y$ in $(L^{\hs}_{N}\frc)_{\le n}$. Let $i_{N}: L^{+}_{N}\frc\incl L^{+}\frc$ be the section that sends a point with coordinates $(a_{0}+a_{1}t+\cdots a_{N-1}t^{N-1}\mod t^{N},\cdots)$ to the point $(a_{0}+a_{1}t+\cdots a_{N-1}t^{N-1},\cdots)\in L^{+}\frc$.  Let $\wt Y=i_{N}(Y)$ and $\wt Y'=i_{N}(Y)$ be the resulting liftings of $Y$ and $Y'$ to $(L^{\hs}\frc)_{\le n}$. Then $Y'$ is a dense constructible subset in the scheme $\wt Y'$ of finite type over $\CC$. We have $\wt Y\subset (L^{\hs}\frc)_{\ov\cO}$.  Our goal is to show that $\wt Y'\subset (L^{\hs}\frc)_{\ov\cO}$,  for then taking image in $L^{+}_{N}\frc$ it implies $Y'=Y$ hence $Y$ is closed in $(L^{\hs}_{N}\frc)_{\le n}$, and hence $(L^{\hs}\frc)_{\ov\cO}\cap(L^{\hs}\frc)_{\le n}$ is closed in $(L^{\hs}\frc)_{\le n}$.

Let $\kappa: L^{+}\frc\to L^{+}\frg$ be the section obtained by applying $L^{+}$ to the Kostant section. Let $V=\kappa(\wt Y)$ and $V'=\kappa(\wt Y')\subset  L^{\hs}\frg$. Let $\cX_{V'}\subset \Gr\times V'$ be the family of affine Springer  fibers for elements in $V'$. We have the evaluation map $\ev:\cX_{V'}\to [\cN/G]$. Let $\cZ=\ev^{-1}(\ov\cO/G)$, then $\cZ$ is a closed subset of $\cX_{V'}$. Let $p: \cX_{V'}\to V'$ be the projection. By assumption, $Y\subset (L^{\hs}\frc)_{\ov\cO}$, every point $v\in V$ can be $LG$-conjugated to $\ov\cO+t\frg\tl{t}$, hence the image of $p|_{\cZ}$  contains $V$. Since $p$ is ind-proper and $\cZ$ is closed in $\cX_{V'}$, the image of $p|_{\cZ}: \cZ\to V'$ is a countable union of closed subsets and $p(\cZ)\supset V$. This implies $p(\cZ)=V'$. In particular,  any $v\in V'$ can be $LG$-conjugated to $\ov\cO+t\frg\tl{t}$, hence the image of $V'$ in $L^{+}\frc$, which is $\wt Y'$, satisfies $\wt Y'\subset (L^{\hs}\frc)_{\ov\cO}$.  This finishes the proof.
\end{proof}

\begin{lemma}\label{l:ineq}
Let $[w]\in \un W$ and $\cO\in\un\cN$. Then the following are equivalent:
\begin{enumerate}
\item $(L^{\hs}\frc)^{sh}_{[w]}\subset (L^{\hs}\frc)_{\ov\cO}$. 
\item $(L^{\hs}\frc)_{[w]}\subset (L^{\hs}\frc)_{\ov\cO}$.
\item $\ov{(L^{\hs}\frc)^{sh}_{[w]}}\cap L^{\hs}\frc\subset (L^{\hs}\frc)_{\ov\cO}$.
\item For any $\g\in L^{\hs}\frg$ of type $[w]$, there exists $\cO'\in\RT_{\min}(\g)$ such that $\cO'\le \cO$.  
\end{enumerate}
When any of these is satisfied, we write $[w]\sqsubset\cO$. Moreover, when $[w]\sqsubset\cO$, we have
\begin{equation}\label{ineq}
\d_{[w]}\ge d_{\cO}.
\end{equation}
\end{lemma}
\begin{proof}
(1)(2)(3) are equivalent since $(L^{\hs}\frc)^{sh}_{[w]}$ is dense in $(L^{\hs}\frc)_{[w]}$ (Lemma \ref{l:cw}), and $(L^{\hs}\frc)_{\ov\cO}$ is closed in $L^{\hs}\frc$ by Lemma \ref{l:cO}. 

We show (4) implies (2). If  $\g\in L^{\hs}\frg$, if $gL^{+}G \in \Gr_{\g,\cO'}$ and $\cO'\subset \ov\cO$, then $\Ad(g^{-1})\g\in (\ov\cO+tL^{+}\frg)\cap L^{\hs}\frg$. Therefore $\chi(\g)=\chi(\Ad(g^{-1})\g)\in(L^{\hs}\frc)_{\ov\cO}$. This being true for all $\g$ of type $[w]$, we conclude that $(L^{\hs}\frc)_{[w]}\subset (L^{\hs}\frc)_{\ov\cO}$.

We show (2) implies (4). For $\g\in L^{\hs}\frg$ of type $[w]$, its image $\chi(\g)\in (L^{\hs}\frc)_{[w]}\subset (L^{\hs}\frc)_{\ov\cO}$. Therefore there exists $\g'\in (\ov\cO+tL^{+}\frg)\cap L^{\hs}\frg$ such that $\chi(\g')=\chi(\g)$. Since both $\g$ and $\g'$ are regular semisimple, there exists $g\in LG$ such that $\Ad(g^{-1})\g=\g'$. Then $gL^{+}G\in \Gr_{\g, \le \cO}$. This implies that some $\cO'\in\RT_{\min}(\g)$ satisfies $\cO'\le \cO$.

We prove the inequality \eqref{ineq}. 
Let $\g$ be shallow of type $[w]$ and $\cO'\in \RT_{\min}(\g)$ such that $\cO'\le \cO$.  Let $\pi: \Fl_{\g}\to \Gr_{\g}$ be the projection. Then the fibers of $\pi$ over $\Gr_{\g,\cO}$ are Springer fibers  $\cB_{e}$ for $e\in \cO$. Therefore
\begin{equation}\label{dg}
\d_{[w]}=\dim\Fl_{\g}\ge\dim\Gr_{\g,\cO}+\dim\cB_{e}\ge \dim \cB_{e}=d_{\cO}.
\end{equation}
\end{proof}

The following transitivity property is immediate from the definition.
\begin{lemma}\label{l:trans} For $[w'],[w]\in \un W$ and $\cO', \cO\in \un\cN$ we have:
\begin{equation}\label{ineq trans}
\mbox{If $[w']\preceq [w]$ and $[w]\sqsubset \cO$, then $[w']\sqsubset \cO$.}
\end{equation}
\begin{equation*}
\mbox{If $[w]\sqsubset \cO$ and $\cO\le \cO'$, then $[w]\sqsubset \cO'$.}
\end{equation*}
\end{lemma}

\begin{defn} For $[w]\in \un W$, let $\RT_{\min}([w])$ be the set of minimal elements in the set
\begin{equation}\label{over w}
\{\cO\in\un\cN| [w]\sqsubset \cO\}.
\end{equation}
\end{defn}
Note that the set \eqref{over w} is non-empty because $(L^{\hs}\frc)_{\ov\cO_{\reg}}=L^{\hs}\frc$ where $\cO_{\reg}$ is the regular nilpotent orbit.

\begin{lemma}\label{l:open}
There is an open dense subset $U_{[w]}\subset (L^{\hs}\frc)^{sh}_{[w]}$ such that if $\g\in L^{\hs}\frg$ and $\chi(\g)\in U_{[w]}$, then $\RT_{\min}(\g)=\RT_{\min}([w])$.
\end{lemma}
\begin{proof}
Let $\cS=\{\cO\in \un\cN|\cO< \cO' \mbox{ for some } \cO'\in \RT_{\min}([w])\}$. By definition, for each $\cO\in \cS$, $(L^{\hs}\frc)^{sh}_{[w]}$ is not contained in $(L^{\hs}\frc)_{\ov\cO}$, hence $V_{\cO}:=(L^{\hs}\frc)^{sh}_{[w]}\cap (L^{\hs}\frc)_{\ov\cO}$ is a proper closed subset. Take $U_{[w]}$ to be the complement of $V_{\cO}$ for all $\cO\in\cS$.
\end{proof}

\begin{lemma}\label{l:sk RT Levi}
Let $M$ be a Levi subgroup of $G$ containing $T$, and $w\in W_{M}$. Then $\io_{M,G}(\RT^{M}_{\min}([w]_{M}))\supset\RT_{\min}([w])$ (see \S\ref{ss:Levi} for notation).
\end{lemma}
\begin{proof} Let $LT_{w}\subset M$ be a loop torus of type $[w]_{M}$. Let $\frc_{M}$ be the Chevalley base for $M$, and $U^{M}_{[w]_{M}}\subset (L^{\hs}\frc_{M})^{sh}_{[w]_{M}}$ be the dense open as in Lemma \ref{l:open} for $M$. Let $\wt U^{M}_{[w]_{M}}\subset L^{++}\frt_{w}$ be its preimage in $L^{++}\frt_{w}$, which is dense open. Similarly,  let $\wt U_{[w]}\subset L^{++}\frt_{w}$ be the preimage of $U_{[w]}$ as in Lemma \ref{l:open}, which is also dense open. Take $\g\in \wt U_{[w]}\cap \wt U^{M}_{[w]_{M}}$. Then $\RT_{\min}(\g)=\RT_{\min}([w])$ and $\RT^{M}_{\min}(\g)=\RT^{M}_{\min}([w]_{M})$. The conclusion then follows from Lemma \ref{l:RT Levi}. 
\end{proof}


\begin{lemma}\label{l:Uw} Let $[w]\in \un W_{ell}$. Suppose $\RT_{\min}([w])$ is a singleton $\{\cO\}$ and $\cO=\Phi([w])$ (i.e., Theorem \ref{th:PhiPsi}\eqref{Phi} holds for $[w]$), then Theorem \ref{th:WtoN}\eqref{Uw} holds for $[w]$.
\end{lemma}
\begin{proof} The existence of $U_{[w]}$ is shown in Lemma \ref{l:open}. It remains to show that for $\g\in (L^{\hs}\frg)_{[w]}^{sh}$, $\cO\in \RT_{\min}(\g)$. 

Since $\cO=\Phi([w])$, by Corollary \ref{c:d ell} we have $\d_{[w]}=d_{\cO}$. Let $\g\in(L^{\hs}\frg)_{[w]}^{sh}$, then $\dim\Fl_{\g}=\d_{[w]}=d_{\cO}$. Since $[w]\sqsubset\cO$ and $\chi(\g)\in (L^{\hs}\frc)_{\ov\cO}$, by Lemma \ref{l:ineq}, there exists $\cO'\in \RT_{\min}(\g)$ such that $\cO'\le \cO$. The inequality \eqref{dg} implies $\d_{[w]}=\dim\Fl_{\g}\ge d_{\cO'}\ge d_{\cO}$. Since $\d_{[w]}=d_{\cO}$, we must have $d_{\cO'}=d_{\cO}$ hence $\cO'=\cO$. 
\end{proof}

\section{The Kazhdan-Lusztig map}\label{s:KL}


The following result is the technical crux for proving the main results for exceptional groups. It uses a theorem of of Bouthier-Kazhdan-Varshavsky \cite{BKV} on the flatness of certain maps between arc spaces, and one of Finkelberg-Kazhdan-Varshavsky \cite{FKV} that settles a conjecture of Lusztig on S-sells.

\begin{theorem}\label{th:Ow} Let $\cO\in \un\cN$ and $[w]=\KL(\cO)$. Then
\begin{enumerate}
\item $\d_{[w]}=d_{\cO}$.
\item We have $(L^{\hs}\frc)_{\ov\cO}=\ov{(L^{\hs}\frc)^{sh}_{[w]}}\cap L^{\hs}\frc$.  In particular, $[w]\sqsubset \cO$. 
\item We have $(L^{\hs}\frc)_{\ov\cO}$ is irreducible and
\begin{equation*}
\codim_{L^{\hs}\frc}(L^{\hs}\frc)_{\ov\cO}=d_{\cO}.
\end{equation*}
See \S\ref{ss:L hs c} for the definition of the codimension of $(L^{\hs}\frc)_{\ov\cO}$ in $L^{\hs}\frc$.
\end{enumerate}
\end{theorem}
\begin{proof}
(1) By the definition of $\KL$, a dense subset of elements $\g\in \cO+tL^{+}\frg$ has type $[w]$, and $\dim\Fl_{\g}=d_{\cO}$ by \cite[Proposition 8.2]{KL}. Therefore $d_{\cO}\ge\d_{[w]}$. To show the equality $\d_{[w]}=d_{\cO}$, it suffices to show that:
\begin{equation}\label{dense O sh}
\mbox{A dense open subset of $\cO+tL^{+}\frg$ consists of shallow elements of type $[w]$.}
\end{equation}

Let $\frI=\Lie \bI\subset L\frg$ be the standard Iwahori subalgebra, and let $\frI^{\hs}=L^{\hs}\frg\cap \frI$. Let $\pi: Z\to \frI^{\hs}$ be the affine Springer fibration restricted to $\frI^{\hs}$, i.e.,
\begin{equation*}
Z=\{(\g, g\bI)\in \frI^{\hs}\times \Fl|\Ad(g^{-1})\g\in \frI^{\hs}\}.
\end{equation*}
Then $Z$ is an affine analogue of the Steinberg variety. Let $\tilW$ be the extended affine Weyl group $\tilW$ of $G$ and let
\begin{equation*}
Z_{v}=\{(\g, g)\in \frI^{\hs}\times (\bI v\bI/\bI)|\Ad(g^{-1})\g\in \frI^{\hs}\}.
\end{equation*}
Then $\{Z_{v}\}_{v\in \tilW}$ is a stratification of $Z$. 
We consider the union $Z_{W}=\cup_{v\in W} Z_{v}$ and denote $\pi_{W}=\pi|_{Z_{W}}: Z_{W}\to \frI^{\hs}$. Note that $Z_{W}=\{(\g, g)\in \frI^{\hs}\times G/B|\Ad(g^{-1})\g\in \frI^{\hs}\}$. From this description, we see that $Z_{W}$ fits into a Cartesian diagram
\begin{equation}\label{fin Z}
\xymatrix{Z_{W}=\cup_{v\in W}Z_{v} \ar[r]^-{\wt\ep} \ar[d]^{\pi_{W}} & \wt\frn=\cup_{v\in W}S_{v}\ar[d]^{\pi_{0}}\\
\frI^{\hs}\ar[r]^-{\ep} & \frn
}
\end{equation}
Here $\ep: \frI^{\hs}\to \frn$ is the reduction mod $t$ map; $\pi_{0}:\wt\frn=\wt\cN|_{\frn}\to \frn$ is the restriction of the Springer resolution $\wt \cN\to \cN$ over $\frn$. It is well-known that $\wt\frn$ is the union of the conormal bundles  $S_{v}$ of the $B$-orbits $BvB/B$ on $G/B$ for $v\in W$, and the map $\ep: Z_{W}\to \wt\frn$ maps $Z_{v}$ onto $S_{v}$ for each $v\in W$. 

Consider the fp locally closed subset of $\frI^{\hs}$
\begin{equation*}
\frI^{\hs}_{\cO}=(\cO\cap\frn+tL^{+}\frg)=\ep^{-1}(\cO\cap \frn)\subset \frI^{\hs}.
\end{equation*}
By the dimension formula for Springer fibers, we know that $\pi_{0}^{-1}(\cO\cap\frn)\subset \wt\frn$ has top dimension, hence contains a dense open subset $S'_{v}\subset S_{v}$ for some $v\in W$. Therefore, by the diagram \eqref{fin Z}, $\pi_{W}^{-1}(\frI^{\hs}_{\cO})=\wt\ep^{-1}(\pi_{0}^{-1}(\cO\cap\frn))$ contains a dense open subset $Z'_{v}=\pi_{W}^{-1}(S'_{v})\subset Z_{v}$. Moreover, since $\pi_{0}$ is flat over $\cO\cap \frn$, $\pi_{0}(S'_{v})$ has top dimension in $\cO\cap \frn$, hence $\pi_{W}(Z'_{v})=\ep^{-1}(\pi_{0}(S'_{v}))$ has top dimension in $\frI^{\hs}_{\cO}$.

Now by \cite[0.8(a)]{FKV}, each $Z_{v}$ contains a dense open subset $Z''_{v}$ whose image in $\frI^{\hs}$ consists of shallow elements of some type $[w_{v}]$ depending on $v$. For our $v$ above, $\pi_{W}(Z'_{v})$ has top dimension in $\frI^{\hs}_{\cO}$, hence the same is true for $\pi_{W}(Z'_{v}\cap Z''_{v})$. This implies that a top-dimensional subset of $\frI^{\hs}_{\cO}$ is shallow of type $[w_{v}]$. Since a dense open subset of $\frI^{\hs}_{\cO}$ has type $[w]$, we conclude that $[w_{v}]=[w]$, which proves that a dense open subset of $\frI^{\hs}_{\cO}$, hence of $\cO+tL^{+}\frg$, is shallow of type $[w]$. This proves \eqref{dense O sh}, which implies $\d_{[w]}=d_{\cO}$.

(2) and (3). Let $(L^{\hs}\frc)_{\ov\cO,[w]}=(L^{\hs}\frc)_{[w]}\cap (L^{\hs}\frc)_{\ov\cO}$. First,  we claim that $(L^{\hs}\frc)_{\ov\cO,[w]}$ is dense in $(L^{\hs}\frc)_{\ov\cO}$. Indeed, by \cite[Last Corollary of \S6]{KL}, there is a dense fp open subset $U\subset \ov\cO+tL^{+}\frg$ contained in $(L^{\hs}\frg)_{[w]}$. Therefore, $\chi(U)$ is dense in $ (L^{\hs}\frc)_{\ov\cO}$. Since $U$ is contained in $(L^{\hs}\frc)_{\ov\cO,[w]}$, the latter is also dense in $(L^{\hs}\frc)_{\ov\cO}$.

Now we claim that 
\begin{equation}\label{le dO}
\codim_{L^{\hs}\frc}(L^{\hs}\frc)_{\ov\cO}\le d_{\cO}.
\end{equation}
Indeed, consider $\chi_{\frI}:=\chi|_{\frI^{\hs}}: \frI^{\hs}\to L^{\hs}\frc$. By the theorem of Bouthier, Kazhdan and Varshavsky \cite[Corollary 7.4.4]{BKV}, $\chi_{\frI}$ is flat. In particular,
\begin{equation}\label{chi flat}
\codim_{L^{\hs}\frc}(L^{\hs}\frc)_{\ov\cO}=\codim_{\frI^{\hs}}\chi_{\frI}^{-1}((L^{\hs}\frc)_{\ov\cO}).
\end{equation}
Note that
\begin{equation*}
\chi_{\frI}^{-1}((L^{\hs}\frc)_{\ov\cO})\supset \frI^{\hs}_{\cO}
\end{equation*}
and the latter has codimension $d_{\cO}=\codim_{\frn}(\frn\cap \ov\cO)$ in $\frI^{\hs}_{\cO}$ (see diagram \eqref{fin Z}). Therefore $\codim_{\frI^{\hs}}\chi_{\frI}^{-1}((L^{\hs}\frc)_{\ov\cO})\le d_{\cO}$, and by \eqref{chi flat}, \eqref{le dO} follows.


On the other hand, by Lemma \ref{l:cw} we have
\begin{equation}\label{codim cw}
\codim_{L^{\hs}\frc}(L^{\hs}\frc)_{[w]}=\codim_{L^{\hs}\frc}(L^{\hs}\frc)^{sh}_{[w]}=\d_{[w]}.
\end{equation}
Now $(L^{\hs}\frc)_{\ov\cO,[w]}$ is dense constructible in $(L^{\hs}\frc)_{\ov\cO}$ so its codimension in $L^{\hs}\frc$ is also $\le d_{\cO}$ by \eqref{le dO}. But $(L^{\hs}\frc)_{\ov\cO,[w]}$ is also contained in $(L^{\hs}\frc)_{[w]}$ which has codimension $\d_{[w]}$ by \eqref{codim cw}. Since $d_{\cO}=\d_{[w]}$, the equality in \eqref{le dO} must hold, and $(L^{\hs}\frc)_{\ov\cO, [w]}$ has the same codimension as $(L^{\hs}\frc)_{[w]}$. Since $(L^{\hs}\frc)^{sh}_{[w]}$ is dense  in $(L^{\hs}\frc)_{[w]}$,  $(L^{\hs}\frc)_{\ov\cO, [w]}$ and $(L^{\hs}\frc)^{sh}_{[w]}$ must share a constructible subset dense in each. Taking closures in $L^{\hs}\frc$ we conclude that
\begin{equation*}
\ov{(L^{\hs}\frc)^{sh}_{[w]}}\cap L^{\hs}\frc=\ov{(L^{\hs}\frc)_{\ov\cO, [w]}}\cap L^{\hs}\frc=(L^{\hs}\frc)_{\ov\cO}.
\end{equation*}

Finally, $(L^{\hs}\frc)_{\ov\cO}$ is irreducible because $(L^{\hs}\frc)^{sh}_{[w]}$ is by Lemma \ref{l:cw}.
\end{proof}

\begin{remark} Theorem \ref{th:Ow} implies \cite[Conjecture 3.5.1]{Spa-ex}.
\end{remark}

\begin{cor}\label{c:KL}
For each $\cO\in \un\cN$, we have $\cO\in \RT_{\min}(\KL(\cO))$. In particular, if Theorem \ref{th:WtoN}\eqref{Uw} holds (i.e., $\RT_{\min}([w])$ is a singleton), then Theorem \ref{th:WtoN}\eqref{KLsec} holds (i.e., $\KL$ is a section of $\RT_{\min}$). 
\end{cor}
\begin{proof}
Let $[w]=\KL(\cO)$. Since $(L^{\hs}\frc)^{sh}_{[w]}\subset (L^{\hs}\frc)_{\ov\cO}$,  $\RT_{\min}([w])$ contains some element $\cO'\le \cO$ by definition. If $\cO'<\cO$, then $(L^{\hs}\frc)^{sh}_{[w]}\subset (L^{\hs}\frc)_{\ov\cO'}$ which has codimension $d_{\cO'}>d_{\cO}$  in $L^{\hs}\frc$ by Theorem \ref{th:Ow}. This is impossible because $(L^{\hs}\frc)^{sh}_{[w]}$ has codimension $\d_{[w]}=d_{\cO}$ by Lemma \ref{l:cw}. Therefore $\cO\in \RT_{\min}([w])$.
\end{proof}

The following proposition characterizes the Kazhdan-Lusztig element $\KL(\cO)$ among all $[w]$ such that $\cO\in\RT_{\min}([w])$.

\begin{prop}\label{p:char KL} Let $\cO\in\un\cO$ and $[w]=\KL(\cO)$.
\begin{enumerate}
\item For $[w']\in \un W$ we have $[w']\sqsubset \cO$ if and only if $[w']\preceq [w]$. In particular, if $\cO\in\RT_{\min}([w'])$, then $[w']\preceq [w]=\KL(\cO)$.
\item If $[w']\sqsubset\cO$ and $\d_{[w']}=d_{\cO}$ (i.e., the equality holds in \eqref{ineq}), then $[w']=\KL(\cO)$.
\end{enumerate}
\end{prop}
\begin{proof}
(1) By Theorem \ref{th:Ow},  $\ov{(L^{\hs}\frc)^{sh}_{[w]}}\cap L^{\hs}\frc=(L^{\hs}\frc)_{\ov\cO}$, hence $[w']\sqsubset \cO$ and  $[w']\preceq [w]$ are equivalent.

(2) If $\d_{[w']}=d_{\cO}$, then $\codim_{L^{\hs}\frc}(L^{\hs}\frc)^{sh}_{[w']}=\d_{[w']}=d_{\cO}=\codim_{L^{\hs}\frc}(L^{\hs}\frc)_{\ov\cO}$. Since $(L^{\hs}\frc)^{sh}_{[w']}\subset (L^{\hs}\frc)_{\ov\cO}$ and the latter is irreducible by Theorem \ref{th:Ow}, we must have $(L^{\hs}\frc)^{sh}_{[w']}$ is dense in $(L^{\hs}\frc)_{\ov\cO}$, therefore it intersects $(L^{\hs}\frc)^{sh}_{[w]}$ since the latter is also dense. This implies $[w']=[w]=\KL(\cO)$.
\end{proof}

\begin{cor}\label{c:two w one O} Suppose $\cO\in\un\cN$, $[w]=\KL(\cO)$ and there exists $[w']\ne [w]$ such that $\RT_{\min}([w'])=\{\cO\}$. Then $\RT_{\min}([w])=\{\cO\}$.
\end{cor}
\begin{proof}
By Corollary \ref{c:KL}, $\cO\in \RT_{\min}([w])$. Since $[w']\sqsubset \cO$, by Proposition \ref{p:char KL}, $[w']\prec[w]$. Therefore, for any $\cO_{1}$ such that $[w]\sqsubset \cO_{1}$, we have $[w']\sqsubset \cO_{1}$ by \eqref{ineq trans}, hence $\cO\le \cO_{1}$ by the definition of $\RT_{\min}([w'])=\{\cO\}$.  This implies that the only minimal element in $\{\cO_{1}|[w]\sqsubset\cO_{1}\}$ is $\cO$, i.e., $\RT_{\min}([w])=\{\cO\}$.
\end{proof}

\section{The skeleton}\label{s:sk}

In this section we introduce a class of subvarieties of the affine Grassmannian called skeleton, which will be used to prove the main results for classical groups.

For each conjugacy class $[w]\in \un W$, we fix a loop torus $LT_{w}\subset LG$ of type $[w]$ under the bijection in \S\ref{ss:loop tori}. 


Recall that $L^{+}T_{w}$ is the parahoric subgroup of $LT_{w}$. Concretely, if we identify $T_{w}$ as the torus over $\CC\lr{t}$ given by the fixed points of $w\z$ on $R_{F_{m}/F}(T\ot_{\CC}F_{m})$ (where $m$ is the order of $w$, $\z$ is a primitive $m$-th root of unity, $F=\CC\lr{t}$ and  $F_{m}=\CC\lr{t^{1/m}}$, see \S\ref{ss:loop tori}), then its Bruhat-Tits group scheme $\cT_{w}$ is the group scheme over $\cO_{F}$ obtained by taking the fiberwise neutral component of the group scheme $R_{\cO_{F_{m}}/\cO_{F}}(T\ot_{\CC}\cO_{F_{m}})^{w\z}$.


\begin{defn}\label{d:skeleton} The {\em skeleton of type $[w]$} is the reduced structure of the fixed point subscheme of $\Gr$ under the left translation by $L^{+}T_{w}$
\begin{equation*}
\cX^{w}:=(\Gr)^{L^{+}T_{w},\red}.
\end{equation*}
\end{defn}

For example, when $[w]=1$ and we take $T_{w}$ to be the split torus $T\ot_{\CC}\CC\lr{t}$, then the reduced structure of $\cX^{w}$ is the same as the reduced structure of $(\Gr)^{T}$, which is natural bijection with $\xcoch(T)$.

Changing $LT_{w}$ to another loop torus in the same $LG$-conjugacy class changes  $\cX^{w}$ by left translation of an element in $LG$. Note that $\cX_{w}$ carries the action of $LT_{w}/L^{+}T_{w}$ by left translation. Alternatively, we may describe points in $\cX^{w}$ as
\begin{equation}\label{Xw}
\cX^{w}(\CC)=\{gL^{+}G|\Ad(g^{-1})L^{+}\frt_{w}\subset L^{+}\frg\}.
\end{equation}

The terminology ``skeleton'' is justified by the following simple observation (in view of \eqref{Xw}).
\begin{lemma}\label{l:Xw in Grg}
For any $\g\in L^{\hs}\frt_{w}$, $\cX^{w}$ is contained in $\Gr_{\g}$.
\end{lemma}

The normalizer $N_{LG}(L^{+}T_{w})=N_{LG}(LT_{w})$ acts on $\cX^{w}$. We have an exact sequence
\begin{equation*}
1\to LT_{w}\to N_{LG}(LT_{w})\to W_{w}\to 1
\end{equation*}
where $W_{w}$ is the centralizer of $w$ in $W$.

\begin{lemma}\label{l:GrO Xw}
For any $\g\in L^{\hs}\frt_{w}$ and $\cO\in \RT_{\min}(\g)$, $\cX^{w}\cap \Gr_{\g, \cO}\ne\vn$.
\end{lemma}
\begin{proof}
If $\cO\in \RT_{\min}(\g)$, then $\Gr_{\g,\cO}$ is closed in $\Gr_{\g}$, hence ind-projective. The action of the connected solvable group $L^{+}T_{w}$ on $\Gr_{\g,\cO}$ therefore has a fixed point, i.e., a point in $\cX^{w}$.
\end{proof}

\begin{lemma}\label{l:sk Levi}
Let $M$ be a Levi subgroup of $G$ containing $T$. Let $w\in W_{M}$ and assume the loop torus $LT_{w}\subset LM$. Then the embedding $\Gr_{M}\incl \Gr$ restricts to an isomorphism $\cX_{M}^{w}\cong\cX^{w}$.
\end{lemma}
\begin{proof}
Since $L^{+}T_{w}$ contains the connected center $A_{M}$ of $M$, and $\Gr_{M}\incl\Gr^{A_{M}}$ is an isomorphism on reduced structures, we have $\cX_{M}^{w}=(\Gr_{M})^{L^{+}T_{w},\red}=\Gr^{L^{+}T_{w},\red}=\cX^{w}$.
\end{proof}

When $LT_{w}$ acts transitively on $\cX^{w}$,  we can deduce some of the conjectures in \S\ref{s:intro}.

\begin{lemma}\label{l:TOfixedpt} Let $w\in W$. Suppose $LT_{w}$ acts transitively on $\cX^{w}$,  then Conjecture \ref{c:min red type} holds for elements of type $[w]$.
\end{lemma}
\begin{proof} Let $\g\in (L^{\hs}\frg)_{[w]}$. Up to $LG$-action we may assume $\g\in L^{\hs}\frt_{w}$.  Suppose $\cO,\cO'\in \RT_{\min}(\g)$, then by Lemma \ref{l:GrO Xw}, we have points $x\in \Gr_{\g,\cO}\cap \cX^{w}$ and $x'\in \Gr_{\g,\cO'}\cap \cX^{w}$. By assumption, $x$ and $x'$ are in the same $LT_{w}$-orbit, which implies that they have the same image in $[\cN/G]$ since $\ev_{\g}$ is invariant under $LT_{w}$. This implies $\cO=\cO'$.
\end{proof}

\begin{prop}\label{p:Cox}
Let $w$ be a Coxeter element of $W$. 
\begin{enumerate}
\item Then $\cX^{w}$ has a unique point in each connected component of $\Gr$. In particular, $LT_{w}$ acts transitively on $\cX^{w}$. 
\item Conjecture \ref{c:min red type} holds for elements of type $[w]$.
\item For any $\g\in (L^{\hs}\frt_{w})^{sh}$, we have $\Gr_{\g}=\cX^{w}$, and $\RT_{\min}(\g)$ consists of the regular nilpotent orbit. In particular, Theorem \ref{th:WtoN}\eqref{Uw} and Conjecture \ref{c:trans} hold for the Coxeter class in $W$.
\end{enumerate}
\end{prop}
\begin{proof} 
For  $\g\in(L^{\hs}\frt_{w})^{sh}$,  the affine Springer fiber $\Gr_{\g}$ has dimension zero by the dimension formula \eqref{dim Sp}, since $\g$ has all root valuation $1/h$. Therefore each point of $\Gr_{\g}$ is fixed by $L^{+}T_{w}$, hence $\Gr_{\g}\subset \cX^{w}$. By Lemma \ref{l:Xw in Grg},  we have $\Gr_{\g}=\cX^{w}$.

Let $\cO_{\reg}$ be the regular nilpotent orbit. By \cite[Corollary 3.7.2]{Ngo}, we have $\Gr_{\g}=\Gr_{\g, \cO_{\reg}}$. In particular,  $\RT(\g)=\{\cO_{\reg}\}=\RT_{\min}(\g)$. This proves (3).  

By \cite[Lemme 3.3.1]{Ngo}, $\Gr_{\g,\cO_{\reg}}$ is transitive under $LT_{w}$. Since $\cX^{w}=\Gr_{\g}=\Gr_{\g,\cO_{\reg}}$, $\cX^{w}$ is also transitive under $LT_{w}$. By Lemma \ref{l:TOfixedpt}, we get (2).

Finally we show (1). The action of $LT_{w}$ on $\cX^{w}$ factors through $LT_{w}/L^{+}T_{w}$. We have $LT_{w}/L^{+}T_{w}\cong \xcoch(T)_{w}=\xcoch(T)/\ZZ R^{\vee}$ ($\ZZ R^{\vee}$ is the coroot lattice), which is the same as the component set of $\Gr$, hence each component of $\Gr$ contains a single point in $\cX^{w}$. 
\end{proof}

\begin{lemma}\label{l:discrete} If $[w]=\KL(\cO)$ for some $\cO\in \un\cN$, and $\g\in(L^{\hs}\frt_{w})^{sh}$, then $\Gr_{\g,\cO}$ is discrete and is contained in $\cX^{w}$.
\end{lemma}
\begin{proof} 
Since $\dim\Fl_{\g}=\d_{[w]}=d_{\cO}$ by (2), the equality in \eqref{dg} holds. Therefore $\dim\Gr_{\g,\cO}=0$ if non-empty. Since $\Gr_{\g,\cO}$ is discrete, its points are fixed by the connected group $L^{+}T_{w}$, therefore $\Gr_{\g,\cO}\subset \cX^{w}$.
\end{proof}

\begin{lemma}\label{l:Xw trans} Let $[w]=\KL(\cO)$ for some $\cO\in \un\cN$. Suppose $LT_{w}$ acts transitively on $\cX^{w}$, then Conjecture \ref{c:trans} hold for $[w]$.
\end{lemma}
\begin{proof} 
Let $\g\in (L^{\hs}\frt_{w})^{sh}$. If $\Gr_{\g,\cO}$ is non-empty, then by Lemma \ref{l:discrete}, it is discrete and $\Gr_{\g,\cO}\subset \cX^{w}$. Since $LT_{w}$ acts transitively on $\cX^{w}$ by assumption, it also acts transitively on $\Gr_{\g,\cO}$, and in fact $\Gr_{\g,\cO}=\cX^{w}$. 
\end{proof}

\begin{prop}\label{p:trans Levi} Let $M$ be a  Levi subgroup of $G$ containing $T$.  Let $w\in W_{M}$, and suppose $[w]=\KL(\cO)$ for some $\cO\in \un\cN$.
\begin{enumerate}
\item There exists $\cO'\in \un \cN_{M}$ such that $\cO'\subset \cO$ and $[w]_{M}=\KL_{M}(\cO')$.
\item If Theorem \ref{th:WtoN}\eqref{Uw} holds for $(M,[w]_{M})$ and Conjecture \ref{c:trans} holds for $(M,[w]_{M}, \cO')$ (for $\cO'$ as in (1)), then Conjecture \ref{c:trans} holds for $(G,[w],\cO)$.
\end{enumerate}
\end{prop}
\begin{proof}
(1) Let $\d_{[w]_{M}}$ denote the function $\d$ in \S\ref{ss:d} for $M$ and the conjugacy class $[w]_{M}$. Let $P$ be a parabolic subgroup of 
$G$ containing $M$ as a Levi subgroup, and let $\frn_{P}$ be its nilpotent radical. For $\g\in  (L^{\hs}\frg)^{sh}_{[w]}\cap L\fm$,   we have
\begin{equation*}
\d_{[w]}-\d_{[w]_{M}}=\val(\det(\ad(\g)|\frn_{P}\ot \CC\lr{t})).
\end{equation*}
On the other hand, by Corollary \ref{c:KL}, $\cO\in \RT_{\min}([w])$. By Lemma \ref{l:sk RT Levi}, $\RT_{\min}([w])\subset \io(\RT^{M}_{\min}([w]_{M}))$. Therefore $\cO\in \io(\RT^{M}_{\min}([w]_{M}))$, hence there exists $\cO'\in \un\cN_{M}$ such that $\cO'\in \RT^{M}_{\min}([w]_{M})$ and $\cO'\subset \cO$. Take any $e\in \cO'$, we have
\begin{equation*}
d_{\cO}-d_{\cO'}=\frac{1}{2}(\dim C_{G}(e)-\dim C_{M}(e))=\dim\ker(\ad(e): \frn_{P}\to \frn_{P}).
\end{equation*}
By Lemma \ref{l:open},  we may choose $\g\in (L^{\hs}\frg)^{sh}_{[w]}\cap L\fm$ such that $\Gr_{M,\g,\cO'}\ne\vn$. Up to $LM$-conjugacy we may assume $\g\equiv e\mod t\fm\tl{t}$ for some $e\in\cO'$. Let $(t^{d_{1}},\cdots, t^{d_{N}})$ be the elementary divisors of $\ad(\g)$ as an endomorphism of $\frn_{P}\ot\CC\tl{t}$ (as a free module over $\CC\tl{t}$). Then $\dim\ker(\ad(e): \frn_{P}\to \frn_{P})$ is the number of $i$ such that $d_{i}\ge1$, and $\val(\det(\ad(\g)|\frn_{P}\ot \CC\lr{t}))$ is $\sum_{i}d_{i}$.   Therefore we get
\begin{equation*}
d_{\cO}-d_{\cO'}\le \d_{[w]}-\d_{[w]_{M}}.
\end{equation*}
Or equivalently
\begin{equation*}
\d_{[w]_{M}}-d_{\cO'}\le \d_{[w]}-d_{\cO}=0.
\end{equation*}
By Lemma \ref{l:ineq}, we always have $\d_{[w]_{M}}\ge d_{\cO'}$. Therefore we must have $\d_{[w]_{M}}=d_{\cO'}$. By Proposition \ref{p:char KL}, $[w]_{M}=\KL_{M}(\cO')$. 

(2) Assume $LT_{w}\subset LM$, and $\g\in (L^{\hs}\frt_{w})^{sh}$ (shallow for $G$). Since Theorem \ref{th:WtoN}\eqref{Uw} and \ref{c:trans} hold for $(M,[w]_{M}, \cO')$, $\Gr_{M,\g,\cO'}$ is non-empty, discrete and transitive under $LT_{w}$. Since $\Gr_{M,\g,\cO'}\subset \Gr_{G,\g,\cO}$, the latter is also non-empty. By Lemma \ref{l:discrete}, $\Gr_{\g,\cO}\subset\cX^{w}$. Since $\cX^{w}=\cX^{w}_{M}$ by Lemma \ref{l:sk Levi}, we have $\Gr_{\g,\cO}\subset \cX^{w}_{M}\subset \Gr_{M,\g}$. This implies
\begin{equation}\label{GrGM}
\Gr_{\g,\cO}=\coprod_{\cO''\subset\cO\cap \cN_{M}}\Gr_{M,\g,\cO''}
\end{equation}
where $\cO''$ runs over nilpotent orbits of $M$ contained in $\cO$. For $\cO''\subset \cO\cap \cN_{M}$, we show that $\Gr_{M,\g,\cO''}=\vn$ unless $\cO''=\cO'$, which would finish the proof for then $\Gr_{\g,\cO}=\Gr_{M,\g,\cO'}$ is transitive under $LT_{w}$.

Suppose $\Gr_{M,\g,\cO''}\ne\vn$ for some $\cO''\in \cO\cap \cN_{M}$. Since $\RT^{M}_{\min}(\g)=\{\cO'\}$, $\cO''\ge \cO'$. If $\cO''\ne \cO'$ then $d_{\cO''}<d_{\cO'}$. Let $e''\in \cO'', e'\in \cO'$, then $\dim C_{M}(e'')<\dim C_{M}(e')$. Let $P$ be a parabolic subgroup of 
$G$ containing $M$ as a Levi subgroup, and let $\frn_{P}$ be its nilpotent radical. Then we have
\begin{eqnarray*}
\dim C_{G}(e'')=\dim C_{M}(e'')+\dim\ker(\ad(e''): \frn_{P}\to \frn_{P}),\\
\dim C_{G}(e')=\dim C_{M}(e')+\dim\ker(\ad(e'): \frn_{P}\to \frn_{P}).
\end{eqnarray*} 
Since $\cO''>\cO'$, $\dim \ker(\ad(e''): \frn_{P}\to \frn_{P})\le \dim\ker(\ad(e'): \frn_{P}\to \frn_{P})$. But we also have $\dim C_{M}(e'')<\dim C_{M}(e')$. Together they imply $\dim C_{G}(e'')<\dim C_{G}(e')$ by the above equalities. However, this contradicts the assumption that $e'$ and $e''$ are both in the same $G$-orbit $\cO$. This shows $\Gr_{M,\g,\cO''}=\vn$. 
\end{proof}

\section{Proofs for type $A, C$ and $G_{2}$}\label{s:AC}

\subsection{What remains to be shown}\label{ss:re}
By Lemma \ref{l:isog}, to prove the theorems stated in \S\ref{s:intro},  we only need to consider the case $G$ is almost simple, and only need to consider one isomorphism class of $G$ among each isogeny class.  It remains to show the following:
\begin{itemize}

\item[(i)] For all $G$ and $[w]$ elliptic non-Coxeter, $\RT_{\min}([w])$ (as defined in Lemma \ref{l:open}) is a singleton and is equal to $\Phi(\cO)$. This implies Theorem \ref{th:PhiPsi}\eqref{Phi} and the full statement of Theorem \ref{th:WtoN} by Lemma \ref{l:Uw} and Corollary \ref{c:KL}. (Note that the case $[w]$ is the Coxeter class is already proved in Proposition \ref{p:Cox}; $\Phi(\Cox)$ is the regular orbit by \cite[8.8]{St-reg}.)

\item[(ii)] Conjecture \ref{c:min red type} for type $A$ and $C$ and $[w]$ elliptic non-Coxeter. This implies the full statement of Theorem \ref{th:AC}, by Corollary \ref{c:reduce to Levi} and Proposition \ref{p:Cox}.

\item[(iii)] Conjectures \ref{c:min red shallow} and \ref{c:trans} for types $A,B,C,D$ and $G_{2}$, and elliptic non-Coxeter $[w]$. This implies the full statement of Theorem \ref{th:trans} by Proposition \ref{p:trans Levi}, Corollary \ref{c:reduce to Levi} and Proposition \ref{p:Cox}.

\item[(iv)] Theorem \ref{th:PhiPsi}\eqref{Psi} for all types and all $[w]$.
\end{itemize}

In this section and the next one we prove the above statements involving  classical groups and $G_{2}$.  Statements (i)(ii)(iii) for type $A$ and $C$ and statements (i)(iii) for type $G_{2}$ will be proved in this section; the rest of the statements will be proved in the next section. 

The strategy is to describe the skeleton $\cX^{w}$ explicitly.  In types $A$ and $C$, we will see that the skeleta are transitive under $LT_{w}$, which allows us to use Lemma \ref{l:TOfixedpt} and Lemma \ref{l:Xw trans} to conclude.  The cases of type $B$ and $D$ are more complicated, and are treated in \S\ref{s:BD}.  

\subsection{Type $A$}\label{ss:A}
Since the Coxeter class is the only elliptic class in type $A$, the statements (i)(ii)(iii) above are vacuous. It only remains to prove Theorem \ref{th:PhiPsi}\eqref{Psi}. However, $\Phi$ is a bijection in this case with inverse $\Psi$. Since $\RT_{\min}=\Phi$,  and $\KL$ is the inverse of $\RT_{\min}$, we have $\KL=\Psi$.

Concretely,  for a conjugacy class $[w]\in S_{n}$ with cycle type $\l$ (a partition of $n$), $\RT_{\min}([w])=\Phi([w])$ is the nilpotent orbit with Jordan type $\l$. 

\subsection{Skeleta in type $C_{n}$}\label{ss:Cn} Let $G=\Sp(V_{0})$ for a symplectic space $V_{0}$ over $\CC$ of dimension $2n$. In this case we check that the condition in Lemma \ref{l:TOfixedpt} holds for $[w]$ elliptic. An elliptic maximal torus $T_{w}$ can be described as follows. Let  $\prod_{i\in I}F_{i}$ be a product of extensions of $F=\CC\lr{t}$ of total degree $n$ (where $I$ is a finite index set) . For each $F_{i}$ let $E_{i}/F_{i}$ be the quadratic extension, then $T_{w}(F)\cong \prod_{i\in I}E_{i}^{\Nm=1}$ (here $\Nm=1$ means the kernel of $\Nm_{E_{i}/F_{i}}$). Under this description, we may identify $V:=V_{0}\ot_{\CC}F$ with
\begin{equation*}
V=\bigoplus_{i\in I}E_{i}
\end{equation*}
Here $E_{i}$ is equipped with the $F$-valued symplectic form
\begin{equation*}
\j{x,y}_{i}=\Tr_{F_{i}/F}\left(d_{i}\frac{x\s_{i} (y)-y\s_{i}(x)}{\pi_{i}}\right)
\end{equation*}
where $\s_{i}\in \Gal(E_{i}/F_{i})$ is the nontrivial element, $\pi_{i}$ is a uniformizer in $\cO^{\s_{i}=-1}_{E_{i}}$, and $d_{i}$ is a generator of the different ideal $\frd_{F_{i}/F}$. Under the symplectic form $\j{\cdot,\cdot}_{i}$ on $E_{i}$,   $\cO_{E_{i}}$ is a self-dual lattice.  We have
\begin{equation*}
L^{+}T_{w}=\prod_{i\in I} (\cO^{1}_{E_{i}})^{\Nm=1}
\end{equation*}
where $\cO^{1}_{E_{i}}=1+\pi_{i}\cO_{E_{i}}$.

The fixed points $\Gr_{G}^{L^{+}T_{w},\red}$ are those self-dual $\cO_{F}$-lattices in $V$ that are stable under the scalar multiplication of $ (\cO^{1}_{E_{i}})^{\Nm=1}$ on the $i$-th factor, for all $i\in I$. Let $\L\subset V$ be such a self-dual lattice. Note that $(\cO^{1}_{E_{i}})^{\Nm=1}$ contains an element of the form $1+\pi'_{i}$ where $\pi'_{i}$ is a uniformizer of $E_{i}$. Therefore,  if $(v_{i})_{i\in I}\in \L\subset V$, then $(v'_{j})_{j\in I}\in \L$ where $v'_{j}=v_{j}$ for $j\ne i$ and $v'_{i}=(1+\pi'_{i})v_{i}$.  Hence $(0,\cdots, \pi'_{i}v_{i},\cdots, 0)\in \L$. This implies that $(0,\cdots, \pi_{i}\cO_{E_{i}} v_{i}, \cdots, 0)\in \L$ for all $i\in I$.

Let $\L\cap E_{i}=\pi_{i}^{a_{i}}\cO_{E_{i}}$. Since $\L$ is self-dual under the symplectic form, and $\cO_{E_{i}}$ is self-dual in $E_{i}$, we have $\pr_{i}(\L)=\pi_{i}^{-a_{i}}\cO_{E_{i}}$ ($\pr_{i}$ is the projection $V\to E_{i}$). In particular, $a_{i}\ge0$. The argument in the previous paragraph implies that $\pi_{i}^{-a_{i}+1}\cO_{E_{i}}\subset \L$. Therefore, we have $\pi_{i}^{-a_{i}+1}\cO_{E_{i}}\subset \pi_{i}^{a_{i}}\cO_{E_{i}}$, which means $-a_{i}+1\ge a_{i}$, or $a_{i}\le1/2$. Therefore we must have $a_{i}=0$ for all $i$, which forces $\L=\oplus_{i\in I}\cO_{E_{i}}$. This shows that for $G=\Sp_{2n}$ and $[w]\in \un W$ elliptic, $\cX^{w}$ is a singleton.

\subsection{Proofs of (i)(ii)(iii) for type $C$}
For elliptic $[w]$, the above discussion shows that $\cX^{w}$ is a singleton, and in particular transitive under $LT_{w}$. Therefore, by Lemma \ref{l:TOfixedpt} and Corollary \ref{l:Xw trans}, we see that $\RT_{\min}([w])$ is a singleton and Conjectures \ref{c:min red type} and \ref{c:trans} hold. 

It remains to show that $\RT_{\min}([w])=\Phi([w])$ for elliptic non-Coxeter $[w]$. For this we need to compute the Jordan type of any $\g\in (L^{\hs}\frt_{w})^{sh}$ on $\L/t\L$, where $\L=\oplus_{i\in I}\cO_{E_{i}}$ is the unique point in $\cX^{w}$. Writing $\g=(\g_{i})_{i\in I}$ with $\g_{i}\in \cO_{E_{i}}^{\s_{i}=-1}$. Since $\g$ is shallow, $\g_{i}$ is a uniformizer of $E_{i}$, hence $\cO_{E_{i}}/t\cO_{E_{i}}$ is single Jordan block under $\g$ of length $2n_{i}=[E_{i}:F]$. Therefore,  if $(2n_{i})$ is the partition of $2n$ given by the cycle type of $[w]$ viewed as an element in $S_{2n}$, $\RT_{\min}(\g)$ consists of the nilpotent class with Jordan type given by the same partition  $(2n_{i})_{i\in I}$. Comparing with \cite[3.7]{Lfromto}, we see that $\RT_{\min}$ coincides with Lusztig's map $\Phi$ on elliptic elements in type $C$.

Using Corollary \ref{c:reduce to Levi}(2), we see that for general conjugacy class $[w]\in \un W$ with cycle type $\l$ (a partition of $2n$), $\RT_{\min}([w])$ is the nilpotent class with Jordan type $\l$.

\subsection{Skeleta in type  $G_{2}$} 
Besides the Coxeter class, there are two elliptic classes in $\un W$ for $G=G_{2}$, namely the class of $\Cox^{2}$ and $\Cox^{3}$. We denote the four nonzero nilpotent orbits of $G_{2}$ by $\cO_{reg}, \cO_{sub}, \cO_{sh}$ and $\cO_{long}$, which have dimension $12,10,8$ and $6$ respectively.

The statement (i) can be checked easily by the relation $\d_{[w]}=d_{\cO}$ if $[w]=\KL(\cO)$. (ii) is vacuous because $\un\cN$ is totally ordered in this case.  

Below we describe $\cX^{w}$ for $w=\Cox^{2}$ and $\Cox^{3}$ and sketch a proof of (iii) in each case. 

\sss{} For $w=\Cox^{2}$ we have $\d_{[w]}=1$ and $[w]=\KL(\cO_{sub})$. For $\g\in (L^{\hs}\frt_{w})^{sh}$, $\Gr_{\g}$ is a union of $\PP^{1}$'s in the affine $\wt E_{6}$-configuration (see \cite[Proposition 7.7(d)]{KL}); it contains one copy of the subregular Springer fiber as the embedding of the $D_{4}$ configuration in the $\wt E_{6}$ configuration.  From this we see that $\Gr_{\g,\cO_{sub}}$ is a singleton, and $\Gr_{\g,\cO_{reg}}$ is the union of three $\AA^{1}$'s. Since $\Gr_{\g,\cO_{reg}}$ is acted transitively by $LT_{w}$, the action of $L^{+}T_{w}$ on each $\AA^{1}$ is transitive. This forces $\cX^{w}\subset \Gr_{\g, \cO_{sub}}$, which is a singleton.

\sss{} For $w=\Cox^{3}=w_{0}$, we have $\d_{[w]}=2$ and $[w]=\KL(\cO_{sh})$. For a $\g\in (L^{\hs}\frt_{w})^{sh}$ (with all root valuations equal to $1/2$), $\Gr_{\g,\cO_{sub}}\cong \PP^{1}-\Gr_{\g,\cO_{sh}}$ where $\Gr_{\g,\cO_{sh}}$ consists of $4$ points.  Since $L^{+}T_{w}$ acts on $\Gr_{\g,\cO_{sub}}$ preserving the $4$ points in $\Gr_{\g,\cO_{sh}}$, the action must be trivial. We conclude that $\Gr_{G}^{L^{+}T_{w},\red}\cong \PP^{1}$. The $4$ points in $\Gr_{\g,\cO_{sh}}$ form a torsor under $\pi_{0}(LT_{w})\ot_{\ZZ}\ZZ/2\cong (\ZZ/2\ZZ)^{2}$.

%
%

\section{Proofs for type $B$ and $D$}\label{s:BD}
\subsection{Skeleta in type $B_{n}$}\label{ss:Bn} Let $G=\SO(V_{0})$ for $\dim_{\CC}V_{0}=2n+1$. Let $LT_{w}$ be an elliptic loop torus in $LG$. In this case the condition in Lemma \ref{l:TOfixedpt} usually does not hold. However the failure of that condition is mild and can be explicitly analyzed.

The elliptic loop tori $LT_{w}$ have the same description as in type $C_{n}$: $LT_{w}\cong\prod_{i\in I}E_{i}^{\Nm=1}$ for quadratic extensions $E_{i}/F_{i}$, where $\prod_{i\in I}F_{i}$ is a separable $F$-algebra of degree $n$.  We will use the notations $\pi_{i}, d_{i},\cdots$ from \S\ref{ss:Cn}.

We may identify the quadratic space $V=V_{0}\ot F$ with
\begin{equation*}
V=F\oplus\bigoplus_{i\in I}E_{i}
\end{equation*}
where $E_{i}$ is equipped with the $F$-valued quadratic form
\begin{equation*}
q_{i}(x)=\Tr_{F_{i}/F}(d_{i}\Nm_{E_{i}/F_{i}}(x))
\end{equation*}
where $d_{i}\in F_{i}$ is a generator of $\frd_{F_{i}/F}$.  Under this form $\cO_{E_{i}}$ is not self-dual but rather dual to $\pi_{i}^{-1}\cO_{E_{i}}$.  To make sure $V$ has a self-dual lattice (i.e., the discriminant of the quadratic form on $V$ has even valuation), the direct summand $F$ in $V$ is equipped with the quadratic form $q_{0}$ so that if $r$ is even, then $\cO_{F}$ is self-dual under $q_{0}$; if $r$ is odd, then $\cO_{F}$ is dual to $t^{-1}\cO_{F}$ under $q_{0}$. The skeleton $\cX^{w}$ is the space of self-dual lattices $\L\subset V$ that are stable under $(\cO^{1}_{E_{i}})^{\Nm=1}$ for all $i\in I$.

We first discuss the case $|I|$ is even. Let $\L\in \cX^{w}$. The same argument as in the type $C$ case shows that $\L$ contains $\pi_{i}\pr_{i}(\L)$  (where $\pr_{i}: V\to E_{i}$ is the projection), viewed as a lattice contained in $E_{i}$. This implies that $\oplus_{i}\cO_{E_{i}}\subset \L \subset F\op(\op_{i}\pi_{i}^{-1}\cO_{E_{i}})$.  Now suppose  $\L\cap F=t^{a_{0}}\cO_{F}$, then $\pr_{0}(\L)=t^{-a_{0}}\cO_{F}$ (where $\pr_{0}: V\to F$ is the projection) since $\cO_{F}$ is self-dual. Suppose $(t^{-a_{0}},\{v_{i}\}_{i\in  I})\in \L$ (so that $v_{i}\in \pi_{i}^{-1}\cO_{E_{i}}$), then $(t^{1-a_{0}},\{tv_{i}\}_{i\in I})\in \L$. Since $tv_{i}\in \cO_{E_{i}}\subset \L$, we conclude $(t^{1-a_{0}},0,\cdots, 0)\in \L$. This implies $1-a_{0}\ge a_{0}$, hence $a_{0}=0$. We conclude that
\begin{equation}\label{range Lam}
\cO_{F}\oplus(\oplus_{i\in I}\cO_{E_{i}})=:\L_{0}\subset \L\subset \L^{\bot}_{0}=\cO_{F}\oplus(\oplus_{i\in I}\pi^{-1}_{i}\cO_{E_{i}}).
\end{equation} 
Conversely, any self-dual lattice $\L$ such that $\L_{0}\subset \L\subset \L_{0}^{\bot}$ is stable under $(\cO^{1}_{E_{i}})^{\Nm=1}$ for all $i\in I$. Therefore $\cX^{w}$ can be identified with the orthogonal Lagrangian Grassmannian $\Lag(\L^{\bot}_{0}/\L_{0})$ for the quadratic space $\L^{\bot}_{0}/\L_{0}$. Note that $\dim\L^{\bot}_{0}/\L_{0}=|I|$ is even, hence $\Lag(\L^{\bot}_{0}/\L_{0})$ has two components, each lying in one component of $\Gr_{G}$.

Now consider the case $|I|$ is odd. The argument is almost identical to the $r$ even case except now $\L_{0}^{\bot}$ has changed slightly
\begin{equation}\label{L0 odd}
\cO_{F}\oplus(\oplus_{i\in I}\cO_{E_{i}})=:\L_{0}\subset \L\subset \L^{\bot}_{0}=t^{-1}\cO_{F}\oplus(\oplus_{i\in I}\pi^{-1}_{i}\cO_{E_{i}}).
\end{equation}
Therefore $\cX^{w}$ can be identified with the orthogonal Lagrangian Grassmannian $\Lag(\L^{\bot}_{0}/\L_{0})$ for the quadratic space $\L^{\bot}_{0}/\L_{0}$ of dimension $|I|+1$.

Our next goal is to calculate $\RT_{\min}([w])$ and show it is the same as $\Phi([w])$. We record here Lusztig's calculation of $\Phi([w])$ for $[w]$ elliptic.

\begin{theorem}[Lusztig \cite{Lfromto}]\label{th:LB} Let $[w]\in \un W(B_{n})$ be an elliptic conjugacy class with negative cycles $n_{1}\ge n_{2}\ge \cdots \ge n_{r}$ (so $r=|I|$ and $[E_{i}:F]=2n_{i}$ in previous notation). Set $n_{0}=\infty$ and $n_{r+1}=n_{r+2}=0$. Then $\Phi([w])$ is the nilpotent orbit in $\so_{2n+1}$ with Jordan blocks of size $\{2n_{i}+\e_{i}\}_{1\le i\le r+1}$ where $\e_{i}\in \{1,0,-1\}$ is defined as
\begin{itemize}
\item $\e_{i}=1$ if $i$ is odd and $n_{i}<n_{i-1}$.
\item $\e_{i}=-1$ if $i$ is even and $n_{i}>n_{i+1}$.
\item $\e_{i}=0$ in all the other cases.
\end{itemize}
\end{theorem}

\subsection{Calculation of $\RT_{\min}([w])$ for $|I|$ even}\label{ss:cal PhiBIeven}
Take a shallow element $\g=(\g_{i})_{i\in I}\in (L^{\hs}\frt_{w})^{sh}$, where $\g_{i}\in \cO^{\s_{i}=-1}_{E_{i}}$ is a uniformizer of $\cO_{E_{i}}$. Let $[E_{i}:F]=2n_{i}$ for $i\in I$.

Consider first the case $|I|$ is even. Let $\L'_{0}=\op \pi_{i}^{-1}\cO_{E_{i}}, \L'_{1}=\op \cO_{E_{i}}$. By \eqref{range Lam}, we have $\L=\L'\op \cO_{F}$, where $\L'_{1}\subset \L'\subset \L'_{0}$. Since $\L/t\L=\L'/t\L'\op \CC$ with $\g$ preserving the two summands, and acting by zero on the $\CC$-factor, it is enough to calculate the Jordan type of $\g$ on $\L'/t\L'$, and adding a length one Jordan block gives the Jordan type of $\ev_{\g}(\L)$.

Let $U=\L'_{0}/\L'_{1}$. This is an $r$-dimensional quadratic space equipped with an orthogonal splitting into lines 
\begin{equation*}
U=\bigoplus_{i\in I}\e_{i}
\end{equation*}
where $\e_{i}=\pi_{i}^{-1}\cO_{E_{i}}/\cO_{E_{i}}$, equipped with a nonzero quadratic form $q=\op_{i\in I} q_{i}$ induced from that on $E_{i}$. We have seen that $L:=\L'/\L'_{1}$ is a Lagrangian in $U$. We have the inclusions
\begin{equation*}
t\L'\subset t\cL'_{0}\subset \L'_{1}\subset \L'
\end{equation*}
such that $\g\L'\subset \L'_{1}, \g(t\L'_{0})\in t\L'$. This gives a three-step filtration on $\L'/t\L'$ with associated graded (from sub to quotient)
\begin{equation}\label{3step}
U/L=L^{*}, \quad \L'_{1}/t\L'_{0}=\bigoplus_{i\in I}\cO_{E_{i}}/t\pi_{i}^{-1}\cO_{E_{i}}, \quad L.
\end{equation}
Here the isomorphism $U/L\cong L^{*}$ uses the quadratic form on $U$. From the shape of $\g$, it has a single Jordan block on each summand $\cO_{E_{i}}/t\pi_{i}^{-1}\cO_{E_{i}}$ of $\L'_{1}/t\L'_{0}$. Therefore the Jordan type of the action of $\g $ on $\L'_{1}/t\L'_{0}$ is the following partition of $2n-r$:
\begin{equation}\label{L'}
(2n_{i}-1)_{i\in I}.
\end{equation}
From the filtration, we see that the Jordan type of $\g$ on $\L'/t\L'$ is obtained from the partition \eqref{L'} by adding $0,1$ or $2$ to each part. To compute the Jordan type exactly, we introduce more notation.

For each $j\in \ZZ$, let $I_{j}=\{i\in I|n_{i}=j\}$. Similarly we define subsets $I_{>j}, I_{\ge j}, I_{<j}, I_{\le j}$. Let  $U_{j}$ be the direct sum of $\e_{i}$ such that $i\in I_{j}$ (i.e., $n_{i}=j$). Then $U_{j}$ is nonzero only for $j$ equal to one of  $n_{1},\cdots, n_{r}$, and $U=\op_{j\in \ZZ}U_{j}$. Similarly we denote $U_{<j}=\op_{j'<j}U_{j'}=\op_{i\in I_{<j}}\e_{i}$, $U_{\ge j}=\op_{j'\ge j}U_{j'}=\op_{i\in I_{\ge j}}\e_{i}$, etc. We also let $q_{j}, q_{\ge j}$ be the restriction of $q$ to $U_{j}, U_{\ge j}$, etc.

Consider the filtration on $L$
\begin{equation*}
L_{\le j}=U_{\le j}\cap L.
\end{equation*}
Then $\Gr_{j}L=L_{\le j}/L_{<j}$ is a subspace of $U_{j}$. Let $\ell_{j}=\dim\Gr_{j}L$.

Let $\dim \Gr_{j}L=\ell_{j}$. Consider the endomorphism $\l_{j}$ of $U_{j}$ given by
\begin{equation*}
\l_{j}:=\g^{2j}: U_{j}=\op_{n_{i}=j}\pi_{i}^{-1}\cO_{E_{i}}/\cO_{E_{i}}\to \op_{n_{i}=j}\pi_{i}^{2j-1}\cO_{E_{i}}/\pi_{i}^{2j}\cO_{E_{i}}=\op_{n_{i}=j}t\pi_{i}^{-1}\cO_{E_{i}}/t\cO_{E_{i}}\cong U_{j}.
\end{equation*}
Then $\l_{j}$ is self-adjoint endomorphism of $U_{j}$. Consider the new quadratic form $Q_{j}$ on $U_{j}$ defined by $x\mapsto (x,\l_{j}x)$ where $(\cdot,\cdot)$ is the symmetric bilinear form associated to $q_{j}=q|_{U_{j}}$. Let $r_{j}$ be the rank of $Q_{j}|_{\Gr_{j}L}$.

\begin{lemma}\label{l:q gen}
The element $\g$ is shallow if and only if for each $j\ge1$, the quadratic form $Q_{j}$ on $U_{j}$ is nondegenerate and is in generic position with the original quadratic form $q_{j}$ (i.e., each $\l_{j}\in \End(U_{j})$ is invertible and has distinct eigenvalues).
\end{lemma}
\begin{proof} Over an algebraic closure of $F$, the eigenvalues of $\g$ is the multiset $R=\{0\}\cup(\cup_{i\in I}\{\s(\g_{i})|\s\in \Gal(E_{i}/F)\})$. The root valuations of $\g$ are the valuations of $\xi-\xi'$ where $\xi$ and $\xi'$ are two distinct elements from $R$. The root valuations achieve minimum if and only if
\begin{enumerate}
\item Each $\g_{i}$ is a uniformizer of $E_{i}$.
\item If $[E_{i}:F]=[E_{j}:F]$, then after identifying $E_{i}$ and $E_{j}$, then $\s(\g_{i})-\g_{j}$ is again a uniformizer of $E_{i}$ for all $\s\in\Gal(E_{i}/F)$. In other words, $\g_{i}^{[E_{i}:F]}$ and $\g_{j}^{[E_{j}:F]}$ (both are of the form $at+\mbox{higher valuation terms}$, for $a\in \CC^{\times}$) have distinct coefficients in front of $t$.
\end{enumerate}
The first condition above is equivalent to saying that each $Q_{j}$ is non-degenerate. The second condition above is equivalent to that for each $j$, $\l_{j}$ have distinct eigenvalues.
\end{proof}

\begin{lemma}\label{l:B ineq} We have the following inequalities:
\begin{enumerate}
\item For any $j\ge1$, we have 
\begin{equation}\label{ell ineq}
\sum_{j'\ge j}\ell_{j'}\ge \frac{1}{2}\dim U_{\ge j}.
\end{equation}
\item If equality in \eqref{ell ineq} holds for some $j\ge1$, and $j'$ is the largest $j'<j$ such that $\ell_{j'}>0$, then $r_{j'}>0$.
\end{enumerate}
\end{lemma}
\begin{proof}
(1) Since $L\subset U$ is a Lagrangian, its image in $U_{\ge j}$ is co-isotropic, hence the inequality \eqref{ell ineq} follows.

(2) If \eqref{ell ineq} holds for $j$, then $L$ has a splitting $L_{<j}\op L_{\ge j}$, with $L_{<j}$ a Lagrangian in $U_{<j}$ and $L_{\ge j}$ a Lagrangian in $U_{\ge j}$. Suppose $L_{<j}\ne0$, then $L_{<j}=L_{\le j'}$ by the definition of $j'$. Then $\Gr_{j'}L\subset L_{j'}$ is coisotropic.    If $\ell_{j'}>\frac{1}{2}\dim U_{j'}$, then $Q_{j'}$ cannot be identically zero on $\Gr_{j'}L$, hence $r_{j'}>0$. If $\ell_{j'}=\frac{1}{2}\dim U_{j'}$, then $\Gr_{j'}L$ is a Lagrangian for $q_{j'}$. Suppose in the contrary that $r_{j'}=0$, then  $\Gr_{i'}L$ is also a Lagrangian for $Q_{j'}$. But this is impossible as $Q_{j'}$ and $q_{j'}$ are in generic position by Lemma \ref{l:q gen}.
\end{proof}

\begin{lemma}\label{l:rk ineq} For  any $N\ge0$, we have
\begin{equation*}
\rk(\g^{N}|_{\L'/t\L'})=\sum_{j\in\ZZ, 2j-1\ge N}\left(2\ell_{j}+(2j-1-N)\dim U_{j}\right)+\begin{cases}r_{N/2}, & N \mbox{ even}\\ 0, & N\mbox{ odd.}\end{cases}
\end{equation*}

\end{lemma}
The proof is a tedious induction argument. We omit it here.

\subsection{Proof of (i) for type $B$, $[w]$ elliptic and $|I|$ even}\label{ss:B shallow}

Let $\g\in(L^{\hs}\frt_{w})^{sh}$. From Lemma \ref{l:B ineq} and Lemma \ref{l:rk ineq},  we have for any $N\ge0$
\begin{equation}\label{min rk}
\rk(\g^{N}|_{\L'/t\L'})\ge \sum_{j\in\ZZ, 2j-1\ge N}(2j-N)\dim U_{j} +\begin{cases}1 & \mbox{if $\dim U_{N/2}\ne0$,} \\ 0 & \mbox{otherwise.}\end{cases}
\end{equation}
The equalities hold for all $N$ only for the following choice of $(\ell_{j}, r_{j})$:
\begin{eqnarray}\label{min ell}
\sum_{j'\ge j}\ell_{j'}=\lceil\frac{1}{2}\dim U_{\ge j}\rceil,\\
\label{min r} 
r_{j}=\begin{cases}1 & \mbox{if $U_{j}\ne0$ and $\dim U_{>j}$ is even}\\ 0 & \mbox{otherwise.}\end{cases}
\end{eqnarray}
in which case $\g|_{\L'/t\L'}$ is in the orbit  $\Phi([w])$ in Theorem \ref{th:LB}. The inequality \eqref{min rk} shows that the orbit of $\g|_{\L'/t\L'}$ always contains $\Phi([w])$ in its closure. 

It remains to show that  choices of $(\ell_{j}, r_{j})$ as in \eqref{min ell} and \eqref{min r} can be achieved for some choice of the Lagrangian $L\subset U$. We will
prove a more precise result  below in Lemma \ref{l:Rg}, which in particular implies such $L$ exists.  \qed

\subsection{Proof of (iii) for type $B$, $[w]$ elliptic and $|I|$ even: preparation}
Let $R_{\g}\subset \Lag(U)$ be the space of Lagrangians satisfying the conditions \eqref{min ell} and \eqref{min r}.  Note that if $R_{\g}\ne\vn$, then $\RT_{\min}(\g)=\{\cO\}$ where $\cO=\Phi([w])$, and $R_{\g}=\Gr_{\g,\cO}$. Clearly $\pi_{0}(LT_{w})\cong \{\pm1\}^{I}$ acts on $U$ by acting on each line $\e_{i}$ by a sign, and preserving the forms $q$ and $Q_{i}$. 

Let $j_{1}>j_{2}>\cdots>j_{s}$ be the set of $j$ such that $\dim U_{j}\ne0$. There is a unique partition of $J=\{j_{1},\cdots, j_{s}\}$ into parts $J(1)\coprod J(2)\coprod\cdots J(a)$ such that
\begin{itemize}
\item Each $J(\a)$ consists of consecutive numbers in $J$;
\item $\sum_{j\in J(\a)}\dim U_{j}$ is even for all $1\le \a\le a$;
\item The partition cannot be further refined and still satisfies the above two conditions.
\end{itemize}
Let $U(\a)=\op_{j\in J(\a)}U_{j}$. Consider the embedding 
\begin{equation*}
\D_{a}:\{\pm1\}^{a}\incl \{\pm1\}^{I}
\end{equation*} 
given by $(\s_{1},\cdots,\s_{a})\mapsto (\d_{i})_{i\in I}$ with $\d_{i}=\s_{\a(i)}$, where $1\le \a(i)\le a$ is such that $\e_{i}\subset U(\a
(i))$ (i.e., $n_{i}\in J(\a(i))$).

\begin{lemma}\label{l:Rg}
Let $\g\in L^{+}\frt_{w}$ be a shallow element. Under the above notations,  the action of $\{\pm1\}^{I}$ on $R_{\g}$ is transitive with stabilizer equal to the image of the embedding $\D_{a}$. 
\end{lemma} 
\begin{proof}
Let $L(\a)=L\cap U(\a)$.  The condition \eqref{min ell} implies that $\sum_{j\in J(\a)} \ell_{j}=\frac{1}{2}\dim U(\a)$ for all $i$. Therefore, $L=\op_{1\le \a\le a}L(\a)$, and each $L(\a)$ is a Lagrangian in $U(\a)$ satisfying the relevant subset of conditions \eqref{min ell} and \eqref{min r}.   This allows us to reduce the argument to the case $U=U(1)$. In the sequel we assume $U=U(1)=U_{j_{1}}\op\cdots\op U_{j_{s}}$. We will show that the action of $\{\pm1\}^{I}/\D\{\pm1\}$ is simple-transitive on $R_{\g}$.

There are two cases: 
\begin{enumerate}
\item $s=1$, and $\dim U=\dim U_{j_{1}}$ is even (and positive). 
\item $s\ge2$, and $\dim U_{j_{1}}$ and $\dim U_{j_{s}}$ are odd, $\dim U_{j_{2}},\dots, \dim U_{j_{s-1}}$ are even.
\end{enumerate}
Consider case (1) first. We denote $j=j_{1}$. In this case all $E_{i}$ have the same degree $2j$ over $F$. Condition \eqref{min r} implies that $L$ (a Lagrangian in $U$) is such that $Q|_{L}$ has rank $1$.  Since $\l$ is shallow, $Q$ and the original quadratic form $q$ on $U$ are in generic position. By Lemma \ref{l:q}(1), the choices of such $L$ is non-empty, and form a torsor under $\Aut(U,q, Q)/\pm1$.

Consider case (2). Let $u_{j}=\dim U_{j}$. Condition \eqref{min ell} implies that $\ell_{j_{1}}=(u_{j_{1}}+1)/2$, $\ell_{j_{2}}=u_{j_{2}}/2,\cdots, \ell_{j_{s-1}}=u_{j_{s-1}}/2$ and $\ell_{j_{s}}=(u_{j_{s}}-1)/2$. 

For $1\le \s\le s$, considering the space $U_{\ge j_{\s}}$ instead of $U$, and  let $X_{\s}$ be the space of minimal co-isotropic subspaces $L\subset U_{\ge j_{\s}}$ such that the analogs of the two conditions \eqref{min ell} and \eqref{min r} hold. In other word, $L\in X_{\s}$ if and only if $\Im(L\to U_{\ge j_{\s'}})$ is minimal co-isotropic for all $1\le \s'\le \s$ and that, setting $L_{j_{\s'}}=(L\cap U_{\le j_{\s}})/(L\cap U_{< j_{\s'}})$,  $Q_{j_{1}}|_{L_{j_{1}}}$ has rank $1$ and $Q_{j_{\s'}}|_{L_{j_{\s'}}}$ is zero for $1<\s'\le \s$. We prove by induction on $\s$ that $\{\pm1\}^{I_{\ge j_{\s}}}/\D\{\pm1\}$ acts simply transitively on $X_{\s}$. 

For $\s=1$, $X_{1}$ is the space of $L\subset U_{j_{1}}$ that is minimal coisotropic for both $q_{j_{1}}$ and $Q_{j_{1}}$. The map $L\mapsto L^{\bot}$ sets a bijection between $X_{1}$ and $Y(q_{j_{1}},Q_{j_{1}})$ in Lemma \ref{l:q}(2). Therefore $X_{1}$ is a torsor under $\{\pm1\}^{I_{ j_{1}}}/\D\{\pm1\}$.

Suppose the statement is proven for $\s\ge1, \s<s-1$, we show that it is also true for $\s+1$. Note by our assumption $\dim U_{j_{\s+1}}$ is even. Let $L\in X_{\s+1}$, and let $L'=L\cap U_{j_{\s+1}}$ and $L''=L/L'\subset U_{\ge j_{\s}}$. By assumption, $\dim L'=\frac{1}{2}u_{j_{\s+1}}$, and $L'$ is a Lagrangian for $Q_{j_{\s+1}}$. On the other hand, $L''$ is minimal coisotropic for $q_{\ge j_{\s}}$ and $L$ is minimal coisotropic for $q_{\ge j_{\s+1}}$, therefore, $q_{j_{\s+1}}|_{L'}$ has rank $\le 1$. Since there does not exist a nonzero subspace of $U_{j_{\s+1}}$ that is isotropic for both $q_{j_{\s+1}}$ and $Q_{j_{\s+1}}$ (because they are in generic position), $q_{j_{\s+1}}|_{L'}$ has rank $1$. Therefore $L'\in X(q_{j_{\s+1}},Q_{j_{\s+1}})$ (see notation of Lemma \ref{l:q}(1)). By  Lemma \ref{l:q}(1), $X(q_{j_{\s+1}},Q_{j_{\s+1}})$ is a torsor under $\{\pm1\}^{I_{j_{\s+1}}}/\D\{\pm1\}$.  Now $L''\in X_{\s}$, and by inductive hypothesis, $X_{\s}$ is a torsor under $\{\pm1\}^{I_{\ge j_{\s}}}/\D\{\pm1\}$. It remains to show that given $L'\in X(q_{j_{\s+1}},Q_{j_{\s+1}})$ and $L''\in X_{\s}$, there are exactly two $L\in X_{\s+1}$ such that $L\cap U_{j_{\s+1}}=L'$ and $\Im(L\to U_{\ge j_{\s}})=L''$ that are interchanged by the diagonal element $\D(-1)\in \{\pm1\}^{I_{j_{\s+1}}}$. Indeed, any such $L$ sits between
\begin{equation*}
L'\op (L'')^{\bot}\subset L\subset (L'+L'^{\bot})\op L''
\end{equation*}
Here the orthogonal complements are taken with respect to $q$. Since $q_{j_{\s+1}}|_{L'}$ has rank $1$, $(L'+L'^{\bot})/L'$ is a $1$-dimensional quadratic space, and so is $L''/L''^{\bot}$. Therefore the choices for $L$ is the same as choices of an isotropic line in the $2$-dimensional quadratic space $(L'+L'^{\bot})/L'\op L''/L''^{\bot}$. From this we see that there are exactly two choices of $L$ that are interchanges by $\D(-1)\in \{\pm1\}^{I_{ j_{\s+1}}}$.

By induction we have proved the statement for $\s=s-1$. Finally we show that the statement holds for $\s=s$. The arugment is almost identical to the previous paragraph. This time we have $L'=L\cap U_{j_{s}}\in Y(q_{j_{s}}, Q_{j_{s}})$ is a torsor under $\{\pm1\}^{I_{j_{s}}}/\D\{\pm1\}$ by Lemma \ref{l:q}(2), $L''\in X_{s-1}$ is a torsor under $\{\pm1\}^{I_{\ge j_{s-1}}}/\D\{\pm1\}$ by what has already been proven by induction. Then 
\begin{equation*}
L'\op (L'')^{\bot}\subset L\subset L'^{\bot}\op L''
\end{equation*}
from which we see there are exactly two choices of $L$ that are interchanges by $\D(-1)\in \{\pm1\}^{I_{ j_{s}}}$. We can now conclude that $X_{s}$ is a torsor under $\{\pm1\}^{I}/\D\{\pm1\}$.
\end{proof}

We have used the following well-known facts about quadrics.
\begin{lemma}\label{l:q} Let $U$ be a finite-dimensional vector space over $\CC$ with two nondegenerate quadratic forms $q_{0}$ and $q_{1}$ in generic position, so that $\Aut(U,q_{0},q_{1})$ is isomorphic to $\{\pm1\}^{\dim U}$.
\begin{enumerate}
\item Suppose $\dim U=2m$. Let $X(q_{0},q_{1})$ be the space of subspaces $L\subset U$ which is a Lagrangian for $q_{0}$ and $q_{1}|_{L}$ has rank $1$. Then $X(q_{0},q_{1})$ is a torsor under $\Aut(U,q_{0},q_{1})/\pm1$. 
\item Suppose $\dim U=2m+1$. Let $Y(q_{0},q_{1})$ be the space of subspaces $M\subset U$ that are maximal isotropic for both $q_{0}$ and $q_{1}$. Then $Y(q_{0},q_{1})$ is a torsor under $\Aut(U,q_{0},q_{1})/\pm1$. 
\end{enumerate}
\end{lemma}
\begin{proof}
(1) Let $(\cdot,\cdot)$ be the symmetric bilinear form associated with $q_{0}$. We may write $q_{1}(x)=(x,\l x)$ for a unique $\l\in\Aut(U)$, self-adjoint with respect to $q_{0}$. Then $q_{0}$ and $q_{1}$ is in generic position if and only if $\l$ is regular semisimple.  Let $L^{i}=L\cap \l^{-1}L\cap\cdots \cap \l^{-i}L$ for $i=0,1\cdots, m$. Clearly $L^{i}\supset L^{i+1}$. Since $q_{1}|_{L}$ has rank $1$, $L^{1}=L\cap \l^{-1}L$ has codimesion $1$ in $L$, i.e., $\dim L^{1}=m-1$. We have $L^{i-1}/L^{i}\subset \l^{-i+1}L/\l^{-i+1}L\cap \l^{-i}L\isom \l L/\l L\cap  L$ (applying $\l^{i}$ in the last step) which has dimension $1$, i.e., $\dim L^{i-1}-\dim L^{i}$ is $1$ or $0$. If $L^{i-1}=L^{i}$ for some $i\ge1$, then $\l(L^{i})=\l L\cap L\cap \cdots\cap\l^{-i+1}L\subset L\cap\cdots\cap \l^{-i+1}L=L^{i-1}=L^{i}$, i.e, $L^{i}$ is stable under $\l$, hence a sum of eigenlines of $\l$. Since $\l$ is regular semisimple, no eigenline of $\l$ is isotropic under $q_{0}$, hence $L^{i}=L^{i-1}$ implies $L^{i-1}=0$. To summarize, we have a full flag of $L$
\begin{equation*}
L=L^{0}\supset L^{1}\supset L^{2}\supset\cdots\supset L^{m-1}\supset L^{m}=0.
\end{equation*}  
Moreover, $L$ is recovered from the line $L^{m-1}=\Span\{v\}$ by
\begin{equation*}
L=\Span\{v,\l v, \cdots, \l^{m-1}v\}.
\end{equation*}
Any nonzero vector $v\in L^{m-1}$ must satisfy $2m-1$ quadratic equations
\begin{equation*}
(v,\l^{i}v)=0, \quad 0\le i\le 2m-2,
\end{equation*}
i.e., $X(q_{0},q_{1})$ is the intersection of the $2m-1$ quadrics in $U$ defined above.  Choosing an eigenbasis $\{e_{1},\cdots, e_{2m}\}$ of $U$ under $\l$, we may write the above equations as diagonal quadratic forms. Note also that $\Aut(U,q_{0},q_{1})\cong\{\pm1\}^{2m}$ with the $i$th factor of $\pm1$ acting on $e_{i}$ by a sign. The conclusion easily follows.

(2) The argument is similar to (1). One shows that $M\in Y(q_{0},q_{1})$ corresponds bijectively to  lines $[v]\in \PP(U)$ satisfying the equations $(v,\l^{i}v)=0$ for $0\le i\le 2m-1$, so that $M=\Span\{v,\l v,\cdots, \l^{m-1}v\}$. 
\end{proof}

\subsection{Proof of (iii) for type $B$, $[w]$ elliptic and $|I|$ even}\label{ss:B trans even} Let $\g\in L^{+}\frt_{w}$ be a shallow element. When  restricted to elliptic classes, $\RT_{\min}$ is injective by Lusztig's description of $\Phi$. Therefore $[w]=\KL(\cO)$ for $\cO=\RT_{\min}([w])$. By \eqref{dg} we see that $\dim\Gr_{\g,\cO}=0$. Since $LT_{w}$ acts on the discrete set $\Gr_{\g,\cO}$, each point of $\Gr_{\g,\cO}$ must be fixed by the connected subgroup $L^{+}T_{w}$. In particular, $\Gr_{\g,\cO}\subset \cX^{w}$.  Now we have $\Gr_{\g,\cO}=R_{\g}$ on which $\pi_{0}(LT_{w})\cong \{\pm1\}^{I}$ acts transitively by Lemma \ref{l:Rg}. \qed

\subsection{Proofs of (i)(iii) for type $B$, $[w]$ elliptic and $|I|$ odd}
Since the argument is similar to the case where $|I|$ is even, we only indicate the difference. Let $U=\L_{0}^{\bot}/\L_{0}$ using notation of \eqref{L0 odd}, then we have a decomposition
\begin{equation*}
U=\e_{0}\op (\bigoplus_{i\in I}\e_{i})
\end{equation*}
where $\e_{0}=t^{-1}\cO_{F}/\cO_{F}$, and $\e_{i}=\pi_{i}^{-1}\cO_{E_{i}}/\cO_{E_{i}}$, equipped with a standard quadratic form $q=\op_{i\in \{0\}\coprod I} q_{i}$. Now $\cX^{w}=\Lag(U)$. If $\L\in \cX^{w}$ corresponds to $L\in \Lag(U)$, then $\L/t\L$ has a three-step filtration with sub to quotient
\begin{equation*}
U/L=L^{*}, \quad \bigoplus_{i\in I}\cO_{E_{i}}/t\pi_{i}^{-1}\cO_{E_{i}}, \quad L.
\end{equation*}
We write $U=\e_{0}\op(\op_{j\in \ZZ}U_{j})$ as in \S\ref{ss:cal PhiBIeven}, and let $U_{0}=\e_{0}$. Then introduce the notations $U_{\le j}, U_{\ge j}, L_{\le j}=L\cap U_{\le j}, \Gr_{j}L$ as in \S\ref{ss:cal PhiBIeven}.  Lemma \ref{l:B ineq} and \ref{l:rk ineq} still hold, with $\L'/t\L'$ replaced by $\L/t\L$ and setting $r_{0}=0$.  The proof of (i) is the same as in the $|I|$ even case, with the same conditions \eqref{min ell} and \eqref{min r}. The analysis of $R_{\g}$ in Lemma \ref{l:Rg} also works exactly the same way in the case $|I|$ is odd, and it shows that $R_{\g}$ is a torsor under $\{\pm1\}^{I}/\Im(\D_{a})$. The rest of the proof of (iii) is identical  to the $|I|$ even case.  \qed

\subsection{Skeleta for type $D_{n}$} Let $G=\SO_{2n}$. An elliptic loop torus $LT_{w}\subset LG$ has the same description as in the type $B_{n}$ case: $LT_{w}\cong \prod_{i\in I}(E_{i})^{\Nm=1}$ for $[E_{i}:F]=2n_{i}$. Moreover, $|I|$ must be even in this case. We will use the same notation as in type $B_{n}$. Then $\cX^{w}$ is the same as in type $B_{n}$ for the same $w$, i.e., $\cX^{w}\cong \Lag(U)$ where $U=\op_{i\in I}\pi_{i}^{-1}\cO_{E_{i}}/\cO_{E_{i}}$ has dimension $|I|$. 

We record Lusztig's calculation of $\Phi([w])$ in type $D_{n}$.
\begin{theorem}[Lusztig \cite{Lfromto}]\label{th:LD} Let $[w]\in \un W(D_{n})$ be an elliptic conjugacy class with negative cycles $n_{1}\ge n_{2}\ge \cdots \ge n_{r}$. Then $r$ is even. Let $n_{0}=\infty$ and $n_{r+1}=0$. Then $\Phi([w])$ is the nilpotent orbit in $\so_{2n}$ with Jordan blocks of size $\{2n_{i}+\e_{i}\}_{1\le i\le r}$ where $\e_{i}\in \{1,0,-1\}$ is defined in the same way as in Theorem \ref{th:LB}.
\end{theorem}

\subsection{Proofs of (i)(iii) in type $D$} The same argument as in \S\ref{ss:B shallow} proves (i) for type $D_{n}$ and elliptic $[w]$. Indeed, if $\l$ is the partition of $2n+1$ given by the Jordan type of $\Phi^{\SO_{2n+1}}_{\RT}([w])$, then the minimal Jordan type of $\g|_{\L/t\L}$ ($\g\in L^{+}\frt_{w}$ is shallow and $\L\in \cX^{w}$) is given by $\l'$ (partition of $2n$) obtained by removing a $1$ from $\l$. Note that the largest part of $\l'$ is odd (by Lusztig's description of $\Phi$ in type $B_{n}$), hence the Jordan type $\l'$ determines a unique the nilpotent orbit in type $D_{n}$. Comparing with Lusztig's description of $\Phi$ in Theorem \ref{th:LD}, we see that $\RT_{\min}([w])=\Phi([w])$.

The same argument as in \S\ref{ss:B trans even} proves (iii) for type $D_{n}$ and elliptic $[w]$. Indeed Lemma \ref{l:Rg} applies in this situation without any change. \qed

\subsection{Proof of Theorem \ref{th:PhiPsi}\eqref{Psi} for classical groups}\label{ss:Psi cl}
Let $\cO\in\un\cN$, $[w]=\Psi(\cO)$ . Since $\Phi=\RT_{\min}$, we have $\RT_{\min}([w])=\cO$. By Proposition \ref{p:char KL}(1), $\KL(\cO)$ is characterized as the maximal element in $\RT^{-1}_{\min}(\cO)=\Phi^{-1}(\cO)$ under $\preceq$.  On the other hand, $\Psi(\cO)$ is characterized as the most elliptic element in $\Phi^{-1}(\cO)$. Therefore, to show that $\KL(\cO)=\Psi(\cO)$, it suffices to show that for any $[w]\in \Phi^{-1}(\cO)$, $[w]\preceq\Psi(\cO)$. We check this for each classical type. 

For type $A$, $\Phi$ is a bijection so there is nothing to prove.

For types $B_{n},C_{n}$ and $D_{n}$, to each conjugacy class of $[w]\in \un W$ we attach a bipartitions $(\a,\b)$ of $n$, where $\a$ is the cycle type of positive cycles and $\b$ is the cycle type for negative cycles in $[w]$ as a permutation of $\{\pm1,\cdots,\pm n\}$. In type $D_{n}$ the bipartition does not completely determine $[w]$: there are two classes attached to each $(\a,\vn)$ when $\a$ has only even parts. This happens precisely when $\cO=\Phi([w])$ consists only of even Jordan blocks. However, in this case $\Phi^{-1}(\cO)$ only has one element, and the statement we want to prove is vacuous. Therefore below we assume that in type $D$, the Jordan blocks of $\cO$ are not all even, so we can represent conjugacy classes $[w],[w']$ by bipartitions.

For $[w]=(\a,\b)$, let $\l(\a,\b)$ be the Jordan type of $\Phi([w])$, which is a partition of $2n$ when $G$ is of type $C_{n}, D_{n}$ and a partition of $2n+1$ when $G$ is of type $B_{n}$. An {\em elementary move} of bipartitions $[w]=(\a,\b)\to (\a',\b')=[w']$ is one in which we remove a part $\a_{i}$ of $\a$ and add two parts $\lceil\frac{\a_{i}}{2} \rceil \lfloor\frac{\a_{i}}{2}\rfloor$ to $\b$ (one part if $\a_{i}=1$). Moreover if $G=C_{n}$, we only allow even $\a_{i}$. 
From the definition of $\Phi$ in \cite{Lfromto}, one sees that for any $[w]\in\Phi^{-1}(\cO)$, there is a sequence of elementary moves  $[w]\to [w_{1}]\to \cdots [w_{t}]=\Psi(\cO)$. Therefore it suffices to show that for any elementary move $[w]\to [w']$ we have $[w]\preceq [w']$. 

Let $[w]\to [w']$ be an elementary move that removes $\a_{i}$ from $\a$. When $\a_{i}>1$, there is a subgroup $H<G$ of isogeny type
\begin{equation*} 
H\sim H_{1}\times H_{2}\sim \begin{cases} D_{\a_{i}}\times B_{n-\a_{i}}, & G\sim B_{n};\\
C_{\a_{i}}\times C_{n-\a_{i}}, & G\sim C_{n};\\
D_{\a_{i}}\times D_{n-\a_{i}}, & G\sim D_{n}. \end{cases}
\end{equation*}
such that $[w]$ and $[w']$ intersect $W_{H}$, and $[w]_{H}=((\a_{i}, \vn)_{H_{1}},  [u]_{H_{2}})$ and $[w']_{H}=( (\vn, \lceil\frac{\a_{i}}{2}\rceil \lfloor\frac{\a_{i}}{2})\rfloor)_{H_{1}}, [u]_{H_{2}})$, where $[u]_{H_{2}}\in\un W_{H_{2}}$. By Lemma \ref{l:order pres}, it suffices to show $[w]_{H}\preceq [w']_{H}$, or show $(\a_{i},\vn)_{H_{1}}\preceq  (\vn, \lceil\frac{\a_{i}}{2}\rceil \lfloor\frac{\a_{i}}{2})\rfloor)_{H_{1}}$ inside $W_{H_{1}}$.

When $G=C_{n}$, only even $\a_{i}$ is allowed. In this case we reduce to show $[v]=(\a_{i},\vn)\preceq  (\vn, \frac{\a_{i}}{2}\frac{\a_{i}}{2})=[v']$ in $W_{H_{1}}=W(C_{\a_{i}})$. Now $\cO_{1}=\Phi_{H_{1}}([v])=\Phi_{H_{1}}([v'])$ has Jordan type $\a_{i}\a_{i}$. Since $[v']$ is elliptic, $\d_{[v']}=d_{\cO_{1}}$ by Corollary \ref{c:d ell}, hence $[v']=\KL_{H_{1}}(\cO_{1})$ by Proposition \ref{p:char KL}(2), and $[v]\preceq [v']$ by Proposition \ref{p:char KL}(1).

When $G=B_{n}$ or $D_{n}$ and $\a_{i}$ is even, we reduce to show that $[v]=(\a_{i},\vn)\preceq  (\vn, \frac{\a_{i}}{2}\frac{\a_{i}}{2})=[v']$ in $W_{H_{1}}=W(D_{\a_{i}})$. Now $\Phi_{H_{1}}([v])=\cO_{1}$ has Jordan type $\a_{i}\a_{i}$, and $\Phi_{H_{1}}([v'])=\cO'_{1}$ has Jordan type $(\a_{i}+1)(\a_{i}-1)$. Therefore $\cO_{1}< \cO'_{1}$ \footnote{This does not contradict the assumption that $\Phi([w])=\Phi([w'])$, see Remark \ref{r:nonLevi red}.}. Since $[v']$ is elliptic we see that $[v']=\KL_{H_{1}}(\cO'_{1})$ (same argument as in type $C$). Now $[v]\sqsubset \cO_{1}<\cO'_{1}$, which implies $ [v]\preceq \KL_{H_{1}}(\cO'_{1})=[v']$ by Proposition \ref{p:char KL}(1).

When $G=B_{n}$ and  $\a_{i}=1$ ($\a_{i}=1$ does not occur in type $D_{n}$), we may find a subgroup  $H<G$ of type $H_{1}\times H_{2}\sim D_{\b_{1}+1}\times B_{n-\b_{1}-1}$, such that $[w]_{H}=((1, \b_{1})_{H_{1}},  [u]_{H_{2}})$ and $[w']_{H}=((\vn, \b_{1}1)_{H_{1}},  [u]_{H_{2}})$. Again we reduce to showing $[v]=(1,\b_{1})\preceq (\vn,\b_{1}1)=[v']$ in $W_{H_{1}}=W(D_{\b_{1}+1})$. Now  $\Phi_{H_{1}}([v])=\Phi_{H_{1}}([v'])\in \un\cN_{H_{1}}$ has Jordan type $(2\b_{1}+1)11$. Since $[v']$ is elliptic, the same argument as in the above cases show that $[v]\preceq [v']$. This finishes the proof of Theorem \ref{th:PhiPsi}\eqref{Psi} for classical groups.


\section{Proofs for exceptional types}\label{s:ex}

Recall $\Phi=\Phi_{G}:\un W\to \un \cN$ is Lusztig's map from \cite{Lfromto}.

\begin{lemma}\label{l:Phi d} Let $G$ be almost simple and of exceptional type. Let $[w], [w']\in \un W$ be two distinct conjugacy classes such that $\Phi([w])=\Phi([w'])$. Assume $[w]$ is elliptic, then $\d_{[w]}<\d_{[w']}$.
\end{lemma}
\begin{proof}
We check the tables in \cite{Lfromto} for the calculation of $\Phi$ in exceptional types. In most cases, from the Carter notation for the conjugacy classes in $W$ we can see $[w']\prec [w]$. For example, if $[w]$ is labelled $4A_{2}$ and $[w']$ is labelled $3A_{2}+A_{1}$, then $[w]$ is the Coxeter class in a Weyl subgroup $W'<W$ isomorphic to $S_{3}^{4}$, and $w'$ can also be chosen from $W'$. Since the Coxeter element is the maximal element under the partial order on $\un W'$, we conclude $[w']\prec [w]$. By Lemma \ref{l:dw ineq}, we have $\d_{[w']}>\d_{[w]}$.

There are a few remaining cases where it is not immediately clear that $[w']\prec [w]$. 
\begin{enumerate}
\item $G=F_{4}$, $([w], [w'])=(A_{3}+\wt A_{1}, B_{2}+A_{1})$.
\item $G=F_{4}$, $([w],[w'])=(4A_{1}, 2A_{1}+\wt A_{1})$.

\item $G=E_{8}$, $([w],[w'])=(A_{7}+A_{1}, D_{5}+A_{2})$.

\item $G=E_{8}$, $([w],[w'])=(2D_{4}(a_{1}), D_{4}(a_{1})+A_{3})$ and  $([w],[w'])=(2D_{4}(a_{1}), 2A''_{3})$.
\item $G=E_{8}$, $([w],[w'])=(2D_{4}, D_{6}(a_{2})+A_{1})$.
\item $G=E_{8}$, $([w],[w'])=(D_{8}(a_{3}), A''_{7})$.
\item $G=E_{8}$, $([w],[w'])=(D_{8}(a_{2}), D_{7}(a_{1}))$.
\item $G=E_{8}$, $([w],[w'])=(D_{8}(a_{1}), D_{7})$.
\end{enumerate}

Case (1). By Corollary \ref{c:d ell}, $\d_{[w]}=d_{\cO}=5$ (here $\cO=\Phi([w])$). We need to show $\d_{[w']}>5$. Let $M<G$ be a Levi subgroup  isogenous to $\Gm\times \Sp_{6}$. Then we may choose $w'\in W_{M}=W(C_{3})$ with cycle types $\ov 2 \ov 1$ (both are negative cycles). Let $LT_{w'}\subset LM$ be a loop torus of type $w'$. We have $\d_{[w']_{M}}=1$. Let $P$ be a parabolic subgroup containing $M$ as a Levi subgroup, then $\frn_{P}\cong V_{14}\op V_{1}$ as a representation of $M$ (with $\dim V_{i}=i$). Since the root valuations of any $\g\in (L^{\hs}\frt_{w'})^{sh}$ are at least $1/4$, we have $\val\det(\ad(\g)|V_{14})\ge \lceil 14/4\rceil=4$, and $\val\det(\ad(\g)|V_{1})\ge1$. Adding up we see that $\d_{[w']}=\d_{[w']_{M}}+\val\det(\ad(\g)|V_{14})+\val\det(\ad(\g)|V_{1})\ge 1+4+1=6>5=\d_{[w]}$.

Case (2). We have $w=w_{0}$ acts by $-1$ on $\frt$, hence $\g\in(L^{\hs}\frt_{w})^{sh}$ has all root valuations equal to $1/2$. From this we see that $\d_{[w]}=(48/2-4)/2=10$. Now $w'$ also has order $2$ hence all root valuations of $\g'\in(L^{\hs}\frt_{w'})^{sh}$ are at least $1/2$ (but not all equal to $1/2$). Moreover $\dim\frt^{w'}=1$. This implies $\d_{[w']}\ge\lceil (48/2-3)/2 \rceil=11>10=\d_{[w]}$.

Case (3). By Corollary \ref{c:d ell}, $\d_{[w]}=d_{\cO}=13$ (here $\cO=\Phi([w])$). We only need to show that $\d_{[w']}>13$. Note that $[w']$ is the Coxeter class in $W_{M}<W$ where $M$ is the Levi subgroup of $G$ by removing one simple root (third from the long end). Let $P$ be the standard parabolic containing $M$. Then as a representation of $M$ (which is isogenous to $\Spin_{10}\times \GL_{3}$),  $\frn_{P}\cong V_{94}\op\wedge^{2}(V_{3})$. Here $\dim V_{94}=94$, and $V_{3}$ means the standard representation of $\GL_{3}$. The root valuations of $\g'\in (L^{\hs}\frt_{w'})^{sh}$ are at least $1/8$. Hence $\val\det(\ad(\g)|V_{94})\ge \lceil 94/8\rceil=12$. On the other hand, $\det(\g'|\wedge^{2}(V_{3}))$ has valuation $2$. Adding up we see that $\d_{[w']}=\val\det(\ad(\g)|V_{94})+
\val\det(\ad(\g)|\wedge^{2}(V_{3}))\ge12+2=14>13=\d_{[w]}$.

For the last 5 cases there is a reductive  subgroup $H<G$ of type $D_{8}$ such that  $[w]$ and $[w']$ both intersect $W_{H}$. By Lemma \ref{l:order pres}, it suffices to show $[w']_{H}\prec [w]_{H}$ in these cases. We use a bipartition $(\a,\b)$ of $8$ to denote a conjugacy class of $H$ (see \S\ref{ss:Psi cl}). We have
\begin{enumerate}
\item[(3)] $[w]_{H}=(\vn,2222)$, $[w']_{H}=(11,222)$ or $(11,2211)$.
\item[(4)] $[w]_{H}=(\vn,3311)$, $[w']_{H}=(2,33)$
\item[(6)] $[w]_{H}=(\vn,44)$, $[w']_{H}$ is one of the classes corresponding to $(8,\vn)$.
\item[(7)] $[w]_{H}=(\vn,53)$, $[w']_{H}=(1,52)$.
\item[(8)] $[w]_{H}=(\vn,62)$, $[w']_{H}=(1,61)$.
\end{enumerate}
In each case,  we find that $\cO'=\Phi_{H}([w']_{H})<\cO=\Phi_{H}([w]_{H})$. Since $[w]_{H}$ is elliptic, $[w]_{H}=\KL_{H}(\cO)$. Now $[w']_{H}\sqsubset \cO'<\cO$, hence $[w']_{H}\prec \KL_{H}(\cO)=[w]_{H}$ by Proposition \ref{p:char KL}(1).
\end{proof}

Recall $\un W_{ell}\subset \un W$ is the set of elliptic classes. Let $\un\cN_{bas}\subset \un \cN$ be the image of $\un W_{ell}$ under $\Phi$ (the subscript $bas$ stands for ``basic'', see \cite{Lfromto}).  By \cite{Lfromto}, $\Phi$ restricts to a bijection $\un W_{ell}\bij \un \cN_{bas}$.

\subsection{Proof of (i) for exceptional $G$} Here (i) refers to the first statement in the beginning of \S\ref{s:AC}. We have already checked that the statement holds for classical types.  By induction on rank of $G$ (assumed exceptional), we assume that Theorem \ref{th:WtoN}\eqref{Uw} and Theorem \ref{th:PhiPsi}\eqref{Phi} hold for proper Levi subgroups of $G$, hence these statements hold for all non-elliptic conjugacy classes in $\un W$ by Corollary \ref{c:reduce to Levi}(2).

For $d\in\ZZ$, let $\un W^{d}$ be the set of conjugacy classes $[w]$ such that $\d_{[w]}=d$.  Let $\un\cN^{d}$ be the set of $\cO\in\un\cN$ such that $d_{\cO}=d$. Let $\un W_{ell}^{d}=\un W^{d}\cap \un W_{ell}$ and  $\un\cN^{d}_{bas}=\un\cN^{d}\cap \un\cN_{bas}$. By Corollary \ref{c:d ell}, $\Phi$ restricts to a bijection $\un W^{d}_{ell}\bij \un \cN^{d}_{bas}$ for all $d$. In particular, $|\un W^{d}_{ell}|=| \un \cN^{d}_{bas}|$.

\begin{claim} For $\cO\in \un\cN_{bas}$, $\KL(\cO)\in \un W_{ell}$.
\end{claim}
\begin{proof}[Proof of Claim]
If not, let $[w']=\KL(\cO)$ which is not elliptic. By induction hypothesis, $\RT_{\min}([w'])=\{\cO\}$, and $\Phi([w'])=\cO$. Since $\cO$ is basic, there exists an elliptic conjugacy class $[w]$ such that $\Phi([w])=\cO$. By Corollary \ref{c:d ell}, $\d_{[w]}=d_{\cO}$. However we also have $\d_{[w']}=d_{\cO}$ since $[w']=\KL(\cO)$. This contradicts Lemma \ref{l:Phi d}.
\end{proof}

By the Claim, the Kazhdan-Lusztig map restricts to a map $\KL^{d}_{ell}: \un\cN^{d}_{bas}\to \un W^{d}_{ell}$. For each $d\in \ZZ$ such that $\un W^{d}_{ell}\ne\vn$, there are two cases.

\begin{enumerate}
\item $\un W^{d}_{ell}$  is a singleton $[w]$. In this case $\un \cN^{d}_{bas}$ is also a singleton $\cO=\Phi([w])$, and we must have $\KL(\cO)=[w]$. This shows that $\cO\in\RT_{\min}([w])$ by Corollary \ref{c:KL}. We show that $\cO$ is the only element in $\RT_{\min}([w])$. For this we have two subcases:
\begin{itemize}
\item If $\cO$ is distinguished, which is equivalent to saying that $\Phi^{-1}(\cO)=\{[w]\}$. If $\cO'\in \RT_{\min}([w])$ and $\cO'\ne \cO$, then $d_{\cO'}\le \d_{[w]}=d_{\cO}=d$. If $d_{\cO'}=d$ we know that $\cO'$ is not basic (since $\un\cN^{d}_{bas}$is a singleton $\cO$) and $[w]=\KL(\cO')$ since $d_{\cO'}=\d_{[w]}$. Since $\cO'$ is not basic, there is a non-elliptic $[w']$ such that $\Phi([w'])=\cO'$, and by induction hypothesis $\RT_{\min}([w'])=\{\cO'\}$. By Corollary \ref{c:two w one O} we have $\RT_{\min}([w])=\{\cO'\}$, contradiction. This forces $d_{\cO'}<d$. By examining the partial order diagram of $G$, we find that all $\cO'$ with $d_{\cO'}<d$ satisfies $\cO'\ge \cO$. This contradicts the definition of $\RT_{\min}([w])$. We conclude that $\RT_{\min}([w])=\{\cO\}$.
\item If $\cO$ is not distinguished, which is equivalent to saying that there exists a non-elliptic $[w']$ such that $\Phi([w'])=\cO$. By induction hypothesis, $\RT_{\min}([w'])=\{\cO\}$. We conclude by applying Corollary \ref{c:two w one O}.
\end{itemize}

\item When $|\un W^{d}_{ell}|=|\un \cN^{d}_{bas}|>1$, the cardinality must be $2$ by examining the tables of nilpotent orbits.  This happens only for  $d=3$ in case $G=F_{4}$ and for $d=4,7,8,9,18,19$ in case $G=E_{8}$. Write $\un W^{d}_{ell}=\{[w_{1}],[w_{2}]\}$, and correspondingly $\un \cN^{d}_{ell}=\{\cO_{1},\cO_{2}\}$ so that $\Phi([w_{i}])=\cO_{i}$ for $i=1,2$. By examining the tables of nilpotent orbits,  we find that at most  one of $\cO_{1}$ and $\cO_{2}$ is distinguished. Let us assume $\cO_{1}$ is {\em not} distinguished. Let $[w_{1}']$ be a non-elliptic class such that $\Phi([w'_{1}])=\cO_{1}$. By examining tables, we see that $[w'_{1}]\preceq[w_{1}]$. By induction hypothesis, $\RT_{\min}([w'_{1}])=\{\cO_{1}\}$. By Corollary \ref{c:two w one O}, $\KL(\cO_{1})$ satisfies $\RT_{\min}(\KL(\cO_{1}))=\{\cO_{1}\}$. This excludes the possibility that $\KL(\cO_{1})=\KL(\cO_{2})$ (for otherwise $\RT_{\min}(\KL(\cO_{1}))$ would contain both $\cO_{1}$ and $\cO_{2}$). We have two possibilities:
\begin{itemize}
\item $\KL(\cO_{1})=[w_{2}], \KL(\cO_{2})=[w_{1}]$. In this case, $\cO_{2}\in\RT_{\min}([w_{1}])$. Since $[w'_{1}]\preceq[w_{1}]\sqsubset\cO_{2}$, hence $[w'_{1}]\sqsubset \cO_{2}$ by Lemma \ref{l:trans}, $\RT_{\min}([w'_{1}])$ has to contain an element $\le \cO_{2}$. However, $\RT_{\min}([w'_{1}])=\{\cO_{1}\}$ and $\cO_{1}$ and $\cO_{2}$ are incomparable. Therefore this is impossible.

\item $\KL(\cO_{1})=[w_{1}], \KL(\cO_{2})=[w_{2}]$. In this case, Corollary \ref{c:two w one O} implies $\RT_{\min}([w_{1}])=\{\cO_{1}\}$. If $\cO_{2}$ is also not distinguished, then the same argument shows that $\RT_{\min}([w_{2}])=\{\cO_{2}\}$. If $\cO_{2}$ is distinguished (this happens only for $d=4,7,8$ in case $G=E_{8}$), we take $\cO'\in \RT_{\min}([w_{2}])$ and assume $\cO'\ne\cO_{2}$. Then $d_{\cO'}\le d$. By examining tables, we see either $\cO'=\cO_{1}$ which would imply $\KL(\cO_{1})=[w_{2}]$ so impossible, or $\cO'>\cO_{2}$ which is also impossible. We conclude that $\RT_{\min}([w_{2}])=\{\cO_{2}\}$ in any case.
\end{itemize}
\end{enumerate}

\subsection{Proof of Theorem \ref{th:PhiPsi}\eqref{Psi} for exceptional groups}\label{ss:Psi ex} 
By Proposition \ref{p:char KL}(2), $\KL(\cO)$ is characterized as the element in $\RT_{\min}^{-1}(\cO)=\Phi^{-1}(\cO)$ with the maximal $\d$-value.  Therefore, to show that $\KL(\cO)=\Psi(\cO)$, it suffices to check that for $[w]\in \Phi^{-1}(\cO)$, $\d_{[w]}$ achieves minimum if and only if $\dim\frt^{w}$ achieves minimum.

For exceptional types we have already checked the above statement when $\Phi^{-1}(\cO)$ contains an elliptic element in Lemma \ref{l:Phi d}. For the remaining cases (i.e., non-basic $\cO$) we go over the tables in \cite[2.2-2.6]{Lfromto2}, and find in most cases when $\Phi^{-1}(\cO)$ contains more than one element, the most elliptic element there is also the largest under the partial order $\preceq$ (using only their names in Carter's notation), hence $\d$-achieves the minimum at the most elliptic element. The only cases that need separate argument are the following, in which we can argue that $\Psi(\cO)$ is the maximal element in $\Phi^{-1}(\cO)$ (which also implies $\Psi(\cO)=\KL(\cO)$ by Proposition \ref{p:char KL}(1)):
\begin{enumerate}
\item $G=F_{4}$, $\Phi^{-1}(\cO)=\{2A_{1}, \wt A_{1}\}$. The classes $[w]=2A_{1}$ (viewed as $D_{2}$) and $[w']=\wt A_{1}$ both intersect the parabolic subgroup $W_{M}<W$ where $M$ is isogenous to $\Gm^{2}\times \SO_{5}$. Using bipartitions, $[w]_{M}$ and $[w']_{M}$ correspond to $(\vn,11)$ and $(1,1)$. Now $\Phi_{M}([w]_{M})=\Phi_{M}([w']_{M})=\cO_{M}$ with Jordan type $311$, and $[w]_{M}$ is elliptic, therefore $[w]_{M}=\KL_{M}(\cO_{M})$ and $[w']_{M}\preceq [w]_{M}$ by Proposition \ref{p:char KL}(1). By Lemma \ref{l:order pres} we have $[w']\preceq [w]$, hence $[w]=\KL(\cO)$ again by Proposition \ref{p:char KL}(1). 

\item $G=E_{7}$, $\Phi^{-1}(\cO)=\{D_{4}(a_{1})+2A_{1}, A_{3}+A_{2}\}$. The classes $[w]=D_{4}(a_{1})+2A_{1}$ (viewing $2A_{1}$ as $D_{2}$) and $[w']=A_{3}+A_{2}$ (viewing $A_{3}$ as $D_{3}$)  both intersect the parabolic subgroup $W_{M}\subset W$ where $M$ is isogenous to $\Gm\times \SO_{12}$. Using bipartitions, $[w]_{M}$ and $[w']_{M}$ correspond to $(\vn,2211)$ and $(3,21)$. Now $\Phi_{M}([w]_{M})=\Phi_{M}([w']_{M})=\cO_{M}$ with Jordan type $5331$, and $[w]_{M}$ is elliptic. The rest of the argument is the same as in case (1).
\end{enumerate}
\qed

\section{Comments}

Our main results have the following consequences on the partial order $\preceq$ on $\un W$.
\begin{cor}\label{c:Phi order} 
\begin{enumerate}
\item Lusztig's map $\Phi$ (or the map $\RT_{\min}$) is order preserving for $(\un W, \preceq)$ and $(\un\cN, \le)$.
\item The Kazhdan-Lusztig map $\KL$ (or Lusztig's map $\Psi$) induces an isomorphism of posets $(\un\cN, \le)\isom (\Im(\KL), \preceq)$.
\end{enumerate}
\end{cor}
\begin{proof}
(1) If $[w_{1}]\preceq [w_{2}]$, letting $\cO_{i}=\Phi([w_{i}])$, we have $[w_{1}]\preceq [w_{2}]\sqsubset \cO_{2}$, hence $[w_{1}]\sqsubset \cO_{2}$. Therefore there is an element in $\RT_{\min}([w_{1}])$ which is $\le\cO_{2}$. Since $\RT_{\min}([w_{1}])=\{\cO_{1}\}$, we have $\cO_{1}\le \cO_{2}$.

(2) If $[w_{i}]=\KL(\cO_{i})$ for $i=1,2$ and $\cO_{1}\le \cO_{2}$, then $[w_{1}]\sqsubset \cO_{1}$ hence $[w_{1}]\sqsubset \cO_{2}$. By Proposition \ref{p:char KL}, this implies $[w_{1}]\preceq [w_{2}]$. This shows that $\KL$ is order-preserving. Since $\KL$ is injective, it  induces an isomorphism of posets $(\un\cN, \le)\isom (\Im(\KL), \preceq)$.
\end{proof}

\subsection{Comparison with another order on $\un W$} In \cite{He}, Xuhua He introduces a partial order $\le_{He}$ as follows:  $[w]\le _{He} [w']$ if and only if there exists minimal length elements $w\in [w]$ and $w'\in [w']$ such that $w\le w'$ under the Bruhat order. In \cite{AHN}, the authors prove that Lusztig's bijection $\Phi: \un W_{ell}\isom \un \cN_{bas}$ is {\em order-reversing} under $\le_{B}$ and the closure order $\le$ on $\un\cN_{bas}$. Combined with Corollary \ref{c:Phi order},  we conclude that the partial orders $\le _{He}$ and $\preceq$ on $\un W_{ell}$ are opposite to each other. 

\subsection{Mininal reduction types for nilpotent orbits} Consider the finite-dimensional analogue of the minimal reduction type. Let $P\subset G$ be a parabolic subgroup with Levi $L$.  For each $\cO\in \un\cN$ we may assign a subset $\RT_{P, \min}(\cO)\subset \un \cN_{L}$ as follows. Let $e\in \cO$, then $\RT_{P,\min}(\cO)$ is the set of minimal nilpotent orbits in the image of the evaluation map $\ev_{P,e}: \cP_{e}\to [\cN_{L}/L]$, where $\cP_{e}=\{gP\in G/P|\Ad(g^{-1})e\in \Lie P\}$ is the partial Springer fiber, and $\ev_{P,e}$ sends $gP$ to the image of $\Ad(g^{-1})e$ in $\Lie L$. However, $\RT_{P,\min}(\cO)$ is not necessarily a  singleton, as the following example shows. 

Consider the case $G=\SL_{4}$, and $P$ corresponds to block sizes $(2,2)$, and $e$ has Jordan type $(3,1)$ (subregular). We know that the Springer fiber $\cB_{e}$ is a chain of three $\PP^{1}$'s in the $A_{3}$-configuration.  Under the natural projection $\cB_{e}\to \cP_{e}$, the two $\PP^{1}$'s at the ends contract to two points $x_{1}, x_{3}\in \cP_{e}\cong \PP^{1}$, while the middle $\PP^{1}$ maps isomorphically onto $\cP_{e}$. Then $\ev_{P,e}(x_{1})$ and $\ev_{P,e}(x_{3})$ are different nilpotent orbits in $L=S(\GL_{2}\times \GL_{2})$ but are both minimal in the image of $\ev_{P,e}$.

\subsection{Affine analogue of special nilpotent orbits}
We have the following analogy:

\begin{center}
\begin{tabular}{|c|c|}
$\un\cN$  & $\un W$\\
\hline
$d_{\cO}\ge a_{\cO}$ & $\d_{[w]}\ge d_{\Phi([w])}$\\
\hline
special nilpotent orbits & image of $\KL$ \\
\hline
special piece & $\Phi^{-1}(\cO)$ or $(L^{\hs}\frg)_{\cO}$
\end{tabular}
\end{center}

We explain how the analogy works row by row.

First, $\un  W$ parametrizes types of topologically nilpotent regular semisimple elements in $L\frg$. It was proposed in \cite{LAffW} that $\un W$ should be thought of as an affine analog of $\un\cN$.

For $\cO\in\un\cN$, Springer correspondence \cite{Spr} gives an irreducible representation $E_{\cO}$ of $W$
\begin{equation*}
E_{\cO}=\cohog{2d_{\cO}}{\cB_{e}}^{A_{e}}
\end{equation*}
where $e\in \cO$ and $A_{e}$ is the component group of the centralizer $C_{G}(e)$. For any irreducible representation $E$ of $W$, Lusztig has attached two functions $a_{E}$ and $b_{E}$ to $E$ \cite[\S4.1]{Lbook}: $a_{E}$ is the lowest power of $q$ in  the formal degree of the $q$-deformation of $E$, and $b_{E}$ is the fake degree, the lowest symmetric power of the reflection representation that contains $E$ as a summand. Lusztig \cite{Lclass} proves the inequality
\begin{equation}\label{ba}
b_{E}\ge a_{E}, \quad \forall E\in \Irr(W).
\end{equation}
For $E=E_{\cO}$, we have $b_{E_{\cO}}=d_{\cO}$. Let $a_{\cO}:=a_{E_{\cO}}$,  then from \eqref{ba} we get
\begin{equation}\label{da}
d_{\cO}\ge a_{\cO}, \quad\forall \cO\in\un\cN.
\end{equation}
We view \eqref{ineq} as an affine analogue of \eqref{da}.

By definition, equality in \eqref{da} holds if and only if $\cO$ is a {\em special nilpotent orbit}. By the analogy between $\un\cN$ and $\un W$,  those $[w]$ such that the equality in \eqref{ineq} holds play the role of special nilpotent orbits. By Proposition \ref{p:char KL}(2), these conjugacy classes in $W$ are exactly those in the image of $\KL: \un\cN\to \un W$. We propose to call the image of $\KL$ {\em $G$-special conjugacy classes} in $W$.  \footnote{The terminology {\em special conjugacy classes in $W$} is defined in \cite{Lfromto2}, and they are in natural bijection with special representations of $W$.}  The Weyl group $W$ of $G$ is at the same time the Weyl group of the Langlands dual group $G^{\vee}$. Therefore, replacing $G$ by $G^{\vee}$ we get the notion of $G^{\vee}$-special conjugacy classes in $W$ as those in the image of $\KL_{G^{\vee}}$.

It is now clear why $\Phi^{-1}(\cO)$ should be thought of as an affine analog of a special piece (a union of nilpotent orbits). We offer a more geometric version below.

\subsection{Affine analogue of special pieces} For a fixed $\cO\in \un\cN$,  we define \footnote{Warning: $(L^{\hs}\frc)_{\cO}$ is not the same as the image of $(\cO+tL^{+}\frg)\cap L^{\hs}\frg$ under $L^{+}\chi:L^{+}\frg\to L^{+}\frc$.}
\begin{equation*}
(L^{\hs}\frc)_{\cO}=(L^{\hs}\frc)_{\ov\cO}-\bigcup_{\cO'< \cO}(L^{\hs}\frc)_{\ov{\cO'}}.
\end{equation*}
By Theorem \ref{th:Ow}, $(L^{\hs}\frc)_{\cO}$ is a fp locally closed subset of $L^{\hs}\frc$.

Clearly the union of $(L^{\hs}\frc)_{\cO}$ for all $\cO\in\un\cN$ is $L^{\hs}\frc$. We make the following observation whose proof is immediate from the definitions.
\begin{lemma}
\begin{enumerate}
\item Let $\g\in L^{\hs}\frg$. Then $\chi(\g)\in (L^{\hs}\frc)_{\cO}$ if and only if $\cO\in \RT_{\min}(\g)$.
\item Conjecture \ref{c:min red type} is equivalent to saying that the $(L^{\hs}\frc)_{\cO}$ are disjoint for various $\cO\in\un\cN$ (in other words, $ \{(L^{\hs}\frc)_{\cO}\}_{\cO\in\un\cN}$ give a partition of $L^{\hs}\frc$). 
\end{enumerate}
\end{lemma}

Therefore, if Conjecture \ref{c:min red type} holds, $(L^{\hs}\frc)_{\cO}$ should be thought of as an affine analog of a special piece in $\cN$.


\subsection{Relating Lusztig's maps and $\RT_{\min}/\KL$ directly?}\label{ss:final} It remains a mystery why the two pairs of maps $(\Phi,\Psi)$ and $(\RT_{\min}, \KL)$ that are defined so differently end up being the same. It is likely that their coincidence can be explained by a wildly ramified Simpson's correspondence, with $(\Phi,\Psi)$ on the Betti side of the picture and $(\RT_{\min}, \KL)$ on the Dolbeault side.  In Minh-Tam Trinh's thesis, he has made deep conjectures relating cohomological invariants of affine Springer fibers on one hand and wild character varieties on the other.  The equality $\Phi=\RT_{\min}$ would be a consequence of special cases of his conjectures.



\subsection*{Acknowledgement} 
At various stages of my career I have been inspired by Springer's work on Springer representations, regular elements in the Weyl group and the geometry of symmetric spaces.  It is my honor to dedicate this paper to the memory of T.A.Springer. 

I have benefitted a lot from conversations with R.Bezrukavnikov, Xuhua He and G.Lusztig on the topics studied in this paper. I would also like to thank Jiu-Kang Yu and the referee for useful suggestions. I'd like to thank Anlong Chua for pointing out a gap in the published version, which lead to the current revision.

\end{document}